\documentclass[11pt]{article}
\usepackage[a4paper,margin=1in]{geometry}
\usepackage[utf8]{inputenc}
\usepackage[T1]{fontenc}

\usepackage{newtxtext,newtxmath}
\usepackage{xcolor}
\usepackage{graphicx}
\usepackage{setspace}
\usepackage{enumitem}
\usepackage{cancel}
\usepackage{ulem}%for sout
\usepackage{mathtools}

\usepackage{amsthm}

\usepackage{booktabs}
\usepackage{multirow}
\usepackage{adjustbox}
\usepackage{float}
\usepackage{makecell}
\usepackage{caption}
\usepackage{subcaption}

\usepackage{algorithm}
\usepackage[noend]{algpseudocode}
\algrenewcommand\algorithmicrequire{\textbf{Input:}}
\algrenewcommand\algorithmicensure{\textbf{Output:}}
\algrenewcommand\algorithmiccomment[1]{\hfill$\triangleright$~#1}
\algnewcommand{\LineComment}[1]{\State $\triangleright$~#1}

\usepackage[numbers]{natbib}
\usepackage[final,nopatch={footnote}]{microtype}

\usepackage[colorlinks=true, linkcolor=blue, citecolor=blue, urlcolor=blue]{hyperref}
\usepackage[nameinlink,noabbrev]{cleveref}

\newcounter{hypcount}

\newenvironment{hyp}{
  \begin{enumerate}[
    label=\textbf{(H\arabic*)},
    ref=H\arabic*,
    leftmargin=*,
    align=left,
    itemsep=0.6ex,
    topsep=0.8ex
  ]
  \setcounter{enumi}{\value{hypcount}}
}{
  \setcounter{hypcount}{\value{enumi}}
  \end{enumerate}
}

\newcommand{\Hyp}[1]{\textup{(\ref{#1})}}

\theoremstyle{plain}

\newtheorem{proposition}{Proposition}[section]
\newtheorem{lemma}{Lemma}[section]

\newtheoremstyle{dashremark}
  {6pt}
  {6pt}
  {\normalfont}
  {}
  {\bfseries\itshape}
  {.}
  {0.5em}
  {\thmname{#1}\thmnumber{~#2}\thmnote{ \textup{(#3)}}}

\theoremstyle{dashremark}
\newtheorem{remark}{Remark}[section]

\crefname{theorem}{theorem}{theorems}
\Crefname{theorem}{Theorem}{Theorems}
\crefname{lemma}{lemma}{lemmas}
\Crefname{lemma}{Lemma}{Lemmas}
\crefname{proposition}{proposition}{propositions}
\Crefname{proposition}{Proposition}{Propositions}
\crefname{remark}{remark}{remarks}
\Crefname{remark}{Remark}{Remarks}
\crefname{corollary}{corollary}{corollaries}
\Crefname{corollary}{Corollary}{Corollaries}
\crefname{definition}{definition}{definitions}
\Crefname{definition}{Definition}{Definitions}

\definecolor{darkred}{RGB}{180,0,0}
\definecolor{jfc}{RGB}{160,0,160}

\title{A posteriori error bounds for the Uncertain Volatility Model}

\author{
Lokman A. Abbas-Turki\thanks{Université Paris Cité and Sorbonne Université, Laboratoire de Probabilités, Statistique et Modélisation (LPSM, UMR CNRS 8001), Paris, France. \texttt{lokmane.abbas\_turki@sorbonne-universite.fr}} 
\and
Jean-François Chassagneux\thanks{ENSAE-CREST and Institut Polytechnique de Paris, Paris, France. \texttt{chassagneux@ensae.fr}. This work benefited from Hi$!$ PARIS and State funding managed by the French National Research Agency (ANR) under the France 2030 program, reference ANR-23-IACL-0005.}
\and
Jean-Philippe Lemor\thanks{BNP Paribas Global Markets, Paris, France. \texttt{jean-philippe.lemor@bnpparibas.com}} 
\and Grégoire Loeper\thanks{BNP Paribas Global Markets, Paris, France. \texttt{gregoire.loeper@bnpparibas.com}} 
\and Simon Sananes\thanks{BNP Paribas Global Markets, Université Paris Cité and Sorbonne Université, Laboratoire de Probabilités, Statistique et Modélisation (LPSM, UMR CNRS 8001), Paris, France. \texttt{sananes@lpsm.paris}} 
}

\date{\today}

\begin{document}

\maketitle

\begin{abstract}
We develop a posteriori primal-dual bounds for numerical approximations of European option prices in the Uncertain Volatility Model. Given a smooth candidate approximation of the value function, a feedback control induced by its Hessian yields a primal lower bound. On the dual side, we derive a representation based on a matrix-valued Gamma field, which provides an upper bound through a nonnegative Hamiltonian penalty. We further show that, when the dual field is generated by the candidate itself, this upper bound admits an equivalent representation in terms of the residual of the associated Black-Scholes-Barenblatt equation. We then study discrete-time approximations of these primal and dual quantities and quantify the corresponding discretization errors. Finally, we investigate their numerical evaluation for candidates obtained by stochastic policy-gradient and physics-informed neural-network methods, and compare the resulting estimates with a martingale dual approach from the stochastic-control literature. The numerical experiments highlight the importance of derivative accuracy, and in particular of second-order information, for obtaining tight a posteriori bounds.
\end{abstract}


\section{Introduction}

In this work, we are interested in computing a posteriori numerical bounds for the price of a European option in the Uncertain Volatility Model (UVM). In this framework, volatility and correlation are not specified by a single model; instead, they are selected dynamically and adversarially within a prescribed uncertainty set. The robust price of the option can therefore be formulated as a stochastic control problem in which the diffusion coefficient of the underlying assets is controlled \cite{ALP95,GV02}. The associated value function solves a fully nonlinear Hamilton-Jacobi-Bellman equation, whose nonlinearity involves a supremum over the admissible volatility controls.

Numerical approximation of the UVM price is challenging even in moderate
dimension. In dimension $d$, the control involves $d$ marginal
volatilities and $d(d-1)/2$ pairwise correlations. The correlation
coefficients cannot be optimized independently: the resulting correlation
matrix must remain positive semidefinite while satisfying the prescribed
componentwise bounds. The optimization of the Hamiltonian is therefore a genuinely constrained matrix optimization problem at every time and state. Moreover, the optimal feedback control need not depend smoothly on the state. Already in dimension one, the extremal volatility is selected according to the sign of the Gamma of the value function, leading to a bang-bang type control. In higher dimension, analogous changes of regime are driven by the matrix of second-order derivatives.

As a particular stochastic control problem, the UVM can be approached by several classes of numerical methods. These include PDE methods \cite{MF17}, 2BSDE methods \cite{FTW11,GZZ15,BEJ19,PWG21}, backward dynamic-programming methods \cite{HPBL21,BHLP22}, and neural PDE methods such as PINNs and related approaches \cite{raissi2019pinns,sirignano2018dgm,cohen2025neural}. The works \cite{GH11,GMZ26,Abb+26} focus specifically on numerical pricing in the UVM. Such methods typically aim at producing an accurate approximation of the value function and, in some cases, an admissible control yielding a lower estimate of the robust price. For applications in which the numerical error itself matters, however, a point estimate or a one-sided estimate is not sufficient: one would like to associate with a given numerical candidate an interval that quantifies its distance from the true UVM value.

Our objective is therefore to construct, from a single candidate approximation of the pricing function, computable primal and dual quantities that bracket the theoretical UVM price, up to the numerical errors introduced when these quantities are themselves approximated.

We first derive continuous-time lower and upper bounds associated with a smooth candidate. The lower bound is obtained from the admissible feedback control induced by the candidate Gamma. For the upper bound, we establish a dual representation of the UVM value in terms of a matrix-valued Gamma field. The corresponding dual penalty is the nonnegative gap between the Hamiltonian and its value under a given admissible control, and hence measures the failure of that control to maximize the Hamiltonian for the prescribed Gamma field. Specializing the dual field to the Gamma of the candidate yields a first upper bound. We then show that the same dual quantity can equivalently be written as the candidate value plus an a posteriori error term involving the residual of the Black-Scholes-Barenblatt equation. Thus the evaluation of the error itself takes the form of a stochastic control problem. Both representations are exact when the candidate coincides with the true value function.

In a second step, we introduce discrete-time approximations of the primal and dual quantities and quantify the corresponding discretization errors using existing results from the stochastic-control literature. We then discuss their practical evaluation for candidate functions produced by different numerical methods. A central point is that an accurate approximation of the option price does not by itself guarantee a tight a posteriori interval: the dual constructions depend explicitly on derivatives of the candidate, and inaccuracies in these derivatives may lead to a substantial duality gap.

In the numerical experiments, we compare our bounds with the martingale dual estimate based on the stochastic-control duality of \cite{HLR16}. In its discretized implementation, this approach introduces a martingale correction constructed from the candidate gradient and performs a pathwise maximization over discretized control sequences. By contrast, our dual quantities exploit the candidate Gamma directly through the UVM Hamiltonian, and therefore make explicit use of second-order information. We observe that, for most of the examples considered, the resulting upper estimates are closer to the reference UVM value and hence lead to tighter numerical intervals. This is an empirical observation rather than a general ordering between
the different dual formulations; the experiments also show that the
tightness of these estimates depends strongly on the accuracy of the
candidate derivatives.

\bigskip
The rest of the paper is organized as follows. In~\Cref{sec:uvm}, we present
the UVM framework. In~\Cref{sec:aposteriori-error-control}, we derive the theoretical a posteriori
error bounds. \Cref{sec:numerical-bounds} is concerned with discrete-time approximations
of the lower and upper bounds associated with a given candidate.
\Cref{sec:numerical-methodology-experiments} describes the numerical implementation and presents numerical
experiments illustrating the behavior of the different primal and dual
estimates.

\paragraph{Notation.}
We denote by $\mathbb S^d$ the space of real symmetric $d\times d$ matrices.
For a subset $E$ of a finite-dimensional normed space and a finite-dimensional
normed space $F$, $\mathcal C^0(E;F)$ denotes the space of continuous
functions from $E$ to $F$, and $\mathcal C^0_{\mathrm{pol}}(E;F)$ its
subspace of functions with polynomial growth, namely those $f$ for which
there exist constants $C>0$ and $m\geq0$ such that
\[
\|f(z)\|_F
\leq
C\bigl(1+\|z\|^m\bigr),
\qquad z\in E.
\]

When the domain is $[0,T]\times D$, polynomial growth always refers to the
spatial variable, uniformly in time. Thus
$f\in\mathcal C^0_{\mathrm{pol}}([0,T]\times D;F)$ if there exist
$C>0$ and $m\geq0$ such that
\[
\|f(t,x)\|_F
\leq
C\bigl(1+\|x\|^m\bigr),
\qquad
(t,x)\in[0,T]\times D.
\]
For functions defined only before maturity, we write
$f\in\mathcal C^0_{\mathrm{loc,pol}}([0,T)\times D;F)$ if
$f\in\mathcal C^0([0,T)\times D;F)$ and there exists $m\geq0$ such that,
for every $\tau<T$, there is a constant $C_\tau>0$ satisfying
\[
\|f(t,x)\|_F
\leq
C_\tau\bigl(1+\|x\|^m\bigr),
\qquad
(t,x)\in[0,\tau]\times D.
\]

We similarly write
$f\in\mathcal C^{1,2}_{\mathrm{pol}}([0,T]\times D;F)$ if
$f\in\mathcal C^{1,2}([0,T]\times D;F)$ and there exist
$C>0$ and $m\geq0$ such that
\[
\|f(t,x)\|_F
+
\|\partial_t f(t,x)\|_F
+
\|D_x f(t,x)\|
+
\|D_x^2 f(t,x)\|
\leq
C\bigl(1+\|x\|^m\bigr),
\]
for all $(t,x)\in[0,T]\times D$. Finally,
$f\in\mathcal C^{1,2}_{\mathrm{loc,pol}}([0,T)\times D;F)$ if
$f\in\mathcal C^{1,2}([0,T)\times D;F)$ and there exists $m\geq0$ such
that, for every $\tau<T$, there is a constant $C_\tau>0$ satisfying
\[
\|f(t,x)\|_F
+
\|\partial_t f(t,x)\|_F
+
\|D_x f(t,x)\|
+
\|D_x^2 f(t,x)\|
\leq
C_\tau\bigl(1+\|x\|^m\bigr),
\]
for all $(t,x)\in[0,\tau]\times D$.

\section{The Uncertain Volatility Model}
\label{sec:uvm}

The Uncertain Volatility Model (UVM) is a framework for robust option
pricing and hedging under joint volatility and correlation uncertainty.
In the multidimensional setting, pricing a European option can be
formulated as a stochastic control problem in which the risky assets form
the controlled state process and the control specifies the volatility and
correlation coefficients.

The UVM was introduced in~\cite{ALP95}; for the
multidimensional formulation considered here, we refer e.g. to~\cite{GV02,Abb+26}.

Fix a dimension $d\geq 1$, a maturity $T\in(0,+\infty)$, and a constant short rate
$r\geq 0$. For each asset $i\in\{1,\ldots,d\}$, let
\[
0<\sigma_{\min}^i\leq \sigma_{\max}^i<+\infty
\]
be the prescribed volatility bounds, and, for each
$1\leq i<j\leq d$, let
\[
-1\leq \rho_{\min}^{ij}\leq \rho_{\max}^{ij}\leq 1
\]
be the prescribed correlation bounds. Denoting by $\mathcal E^d$ the
elliptope of $d$-dimensional correlation matrices, we define the set of
admissible volatility coefficients by
\[
\begin{aligned}
A \coloneqq
\Big\{
a\in\mathbb R^{d\times d}:\;&
aa^\top
=
\operatorname{diag}(\sigma)\rho\operatorname{diag}(\sigma),\\
&
\sigma\in\prod_{i=1}^d[\sigma_{\min}^i,\sigma_{\max}^i],
\quad
\rho\in\mathcal E^d,\quad
\rho^{ij}\in[\rho_{\min}^{ij},\rho_{\max}^{ij}],
\quad
\forall\,1\leq i<j\leq d
\Big\}.
\end{aligned}
\]
Observe that it is a compact subset of
$\mathbb R^{d\times d}$.

For every $a\in A$, the corresponding volatility vector and correlation
matrix are uniquely recovered from $aa^\top$ through
\[
\sigma^i=\sqrt{(aa^\top)^{ii}},
\qquad
\rho^{ij}
=
\frac{(aa^\top)^{ij}}{\sigma^i\sigma^j},
\qquad
1\leq i<j\leq d,
\]
which is well-defined since $\sigma_{\min}^i>0$ for every $i$. Conversely,
an admissible pair $(\sigma,\rho)$ does not in general determine a unique
volatility coefficient $a$, since different matrix factorizations may yield
the same covariance matrix
$\operatorname{diag}(\sigma)\rho\operatorname{diag}(\sigma)$. No particular
factorization is selected in the continuous-time formulation: the local
covariance of the controlled dynamics and the associated Hamiltonian introduced below depend on $a$ only through $aa^\top$.

Let $g:(0,+\infty)^d\to\mathbb R$ denote the payoff of a European option
with maturity $T$. We impose the following assumptions throughout on the set of admissible volatilities and the payoff function.
\begin{hyp}
\item \label{hyp:payoff}
The function $g$ is locally Lipschitz and $g$ and
$\nabla_x g$, defined a.e., have polynomial growth.
\item \label{hyp:ellipticity} The set $A$ is non empty and for every $a\in A$, the matrix $aa^\top$ is positive definite.
\end{hyp}

\noindent For each $t\in[0,T]$, we consider the canonical reference probability
system
\[
\big(\Omega^t,(\mathcal F_s^t)_{s\in[t,T]},\mathbb P^t,W^t\big),
\]
where
\[
\Omega^t
\coloneqq
\big\{
\omega\in\mathcal C^0([t,T];\mathbb R^d):\omega_t=0
\big\},
\]
$W^t$ is the canonical process on $\Omega^t$, $\mathbb P^t$ is Wiener
measure, and $(\mathcal F_s^t)_{s\in[t,T]}$ is the
$\mathbb P^t$-augmentation of the canonical filtration generated by $W^t$.
We denote by
\[
\mathcal A^t
\coloneqq
\Big\{
\alpha=(\alpha_s)_{s\in[t,T]}:
\alpha
\text{ is }(\mathcal F_s^t)_{s\in[t,T]}\text{-progressively measurable
and $A$-valued}
\Big\}
\]
the set of admissible controls.

For $(t,x)\in[0,T]\times(0,+\infty)^d$ and
$\alpha\in\mathcal A^t$, let $X^{\alpha,t,x}$ be the solution of
\begin{equation}\label{eq:UVM-multi-dynamics}
\begin{cases}
dX_s^{\alpha,t,x}
=
rX_s^{\alpha,t,x}\,ds
+
\operatorname{diag}\big(X_s^{\alpha,t,x}\big)\alpha_s\,dW_s^t,
& s\in[t,T],\\
X_t^{\alpha,t,x}=x.
\end{cases}
\end{equation}
The UVM value function is
\begin{equation}
V(t,x)
\coloneqq
\sup_{\alpha\in\mathcal A^t}
\mathbb E^{\mathbb P^t}
\left[
e^{-r(T-t)}g\big(X_T^{\alpha,t,x}\big)
\right].
\end{equation}
We also introduce the Black-Scholes-Barenblatt (BSB) equation
\begin{equation}\label{eq:BSB-multi}
\begin{cases}
\begin{aligned}
&\partial_tV(t,x)
+r\big(x^\top\nabla_xV(t,x)-V(t,x)\big)\\
&\qquad
+\frac12\sup_{a\in A}
G\big(x,\nabla_x^2V(t,x),a\big)
=0,
\end{aligned}
& (t,x)\in[0,T)\times(0,+\infty)^d,\\[3pt]
V(T,x)=g(x),
& x\in(0,+\infty)^d,
\end{cases}
\end{equation}
where, for
$(x,\gamma,a)\in(0,+\infty)^d\times\mathbb S^d\times A$,
\begin{align*}
G(x,\gamma,a)
&\coloneqq
\operatorname{Tr}\left[
\gamma\,
\operatorname{diag}(x)\,
aa^\top\,
\operatorname{diag}(x)
\right]\\
&=
\operatorname{Tr}\left[
\gamma\,
\operatorname{diag}(x\odot\sigma)\,
\rho\,
\operatorname{diag}(x\odot\sigma)
\right]\\
&=
\sum_{i,j=1}^d
x^ix^j\sigma^i\sigma^j\rho^{ij}\gamma^{ij},
\end{align*}
with $\sigma$ and $\rho$ denoting the volatility vector and correlation
matrix associated with $aa^\top$. Finally, for $\gamma\in\mathbb S^d$, we define the Hamiltonian
\begin{equation}\label{eq:support-property}
H^\star[\gamma]
\coloneqq
\sup_{a\in A}\operatorname{Tr}[\gamma aa^\top].
\end{equation}
Then, for every $(x,\gamma)\in(0,+\infty)^d\times\mathbb S^d$,
\[
\frac12\sup_{a\in A}G(x,\gamma,a)
=
H^\star\left[
\frac12\operatorname{diag}(x)\gamma\operatorname{diag}(x)
\right].
\]

Under~\Hyp{hyp:payoff}, the strong formulation considered above coincides
with the weak formulation of the UVM studied in~\cite{GV02}, see
\cite[Proposition~2.1]{Abb+26}. Therefore, by~\cite[Theorem~7]{GV02}, $V$ is the unique
viscosity solution of~\eqref{eq:BSB-multi} with polynomial growth. Moreover, under~\Hyp{hyp:payoff} and~\Hyp{hyp:ellipticity}, by~\cite[Theorem~11]{GV02},
\[
V\in
\mathcal C^0_{\mathrm{pol}}
\big([0,T]\times(0,+\infty)^d;\mathbb R\big)
\cap
\mathcal C^{1,2}_{\mathrm{loc,pol}}
\big([0,T)\times(0,+\infty)^d;\mathbb R\big).
\]
In particular, $V$ is a classical solution of~\eqref{eq:BSB-multi} on
$[0,T)\times(0,+\infty)^d$.

We will repeatedly use the following standard estimates for the controlled state process.

\begin{lemma}
\label{lem:state-estimates}
For every $p>0$, there exists a constant $C_p>0$ such that, for every
$t\in[0,T]$, $x,z\in(0,+\infty)^d$, and $\alpha\in\mathcal A^t$, the
solution $X^{\alpha,t,x}$ of~\eqref{eq:UVM-multi-dynamics} satisfies
\begin{equation}
\label{eq:moment-estimate}
\mathbb E^{\mathbb P^t}
\left[
\sup_{u\in[t,T]}
|X_u^{\alpha,t,x}|^p
\right]
\leq
C_p(1+|x|^p),
\end{equation}
and
\begin{equation}
\label{eq:state-stability}
\mathbb E^{\mathbb P^t}
\left[
\sup_{u\in[t,T]}
|X_u^{\alpha,t,x}-X_u^{\alpha,t,z}|^p
\right]
\leq
C_p|x-z|^p.
\end{equation}
Moreover, for every $h\in[0,T-t]$,
\begin{equation}
\label{eq:state-increment}
\mathbb E^{\mathbb P^t}
\left[
|X_{t+h}^{\alpha,t,x}-x|^p
\right]
\leq
C_p(1+|x|^p)h^{p/2}.
\end{equation}
The constant $C_p$ is independent of $t$, $x$, $z$, and $\alpha$.
\end{lemma}

The proof is given in Appendix~\ref{app:technical}.

\section{A posteriori error control}
\label{sec:aposteriori-error-control}

We now derive a posteriori bounds for smooth approximations of the UVM
value function. We first establish a dual representation of $V$ in terms
of matrix-valued fields. We then specialize this representation to the
field generated by the Hessian of a smooth candidate. This yields a dual
upper bound, while the same candidate induces a natural feedback control
and hence a primal lower bound.

\subsection{Dual representation}
\label{subsec:duality}

For a function $v$ which is twice differentiable in space, define
\[
\Gamma_v(t,x)
\coloneqq
\frac12
\operatorname{diag}(x)
\nabla^2_{xx}v(t,x)
\operatorname{diag}(x).
\]
In dimension one,
\[
2\Gamma_v(t,x)
=
x^2\partial_{xx}^2v(t,x),
\]
which is commonly referred to as the cash gamma or dollar gamma. Thus our normalization corresponds to one half of the usual cash gamma in dimension one. The factor $1/2$ is natural in
second-order hedging formulas, where
$\frac12x^2\partial_{xx}^2v$ multiplies the variance term, see, e.g.~\cite{CM98}.
In particular, the regularity established in \Cref{sec:uvm} implies that 
\[
\Gamma_V
\in
\mathcal C^0_{\mathrm{loc,pol}}
\big([0,T)\times (0,+\infty)^d;\mathbb S^d\big).
\]

For $\Gamma
\in
\mathcal C^0_{\mathrm{loc,pol}}
\big([0,T)\times (0,+\infty)^d;\mathbb S^d\big)$,
define
\begin{align*}
\mathcal D(\Gamma)(t,x)
\coloneqq
\sup_{\alpha\in\mathcal A^t}
\mathbb E^{\mathbb P^t}\Bigg[
&e^{-r(T-t)}
g(X_T^{\alpha,t,x})
\nonumber\\
&+
\int_t^T
e^{-r(s-t)}
\Big(
H^\star[\Gamma(s,X_s^{\alpha,t,x})]
-
\operatorname{Tr}\big[
\Gamma(s,X_s^{\alpha,t,x})
\alpha_s\alpha_s^\top
\big]
\Big)\,ds
\Bigg].
\end{align*}
Note that by the definition of $H^\star$,
\begin{equation}\label{eq:duality-gap-nonnegative}
H^\star[\gamma]
-
\operatorname{Tr}[\gamma aa^\top]
\geq0,
\qquad
(\gamma,a)\in\mathbb S^d\times A.
\end{equation}
Since the terminal payoff is integrable, $\mathcal D(\Gamma)(t,x)$ is
therefore well defined with values in $\mathbb R\cup\{+\infty\}$.

\begin{proposition}[Dual representation]
\label{prop:uvm-convex-duality}
For every $(t,x)\in[0,T]\times (0,+\infty)^d$,
\begin{equation}\label{eq:dual-representation}
V(t,x)
=
\inf_{\Gamma\in
\mathcal C^0_{\mathrm{loc,pol}}
([0,T)\times (0,+\infty)^d;\mathbb S^d)}
\mathcal D(\Gamma)(t,x)
=
\mathcal D(\Gamma_V)(t,x).
\end{equation}
\end{proposition}

\begin{proof}
Let
$
\Gamma
\in
\mathcal C^0_{\mathrm{loc,pol}}
\big([0,T)\times (0,+\infty)^d;\mathbb S^d\big).
$
For every $\alpha\in\mathcal A^t$,
\eqref{eq:duality-gap-nonnegative} yields
\begin{align*}
\mathbb E^{\mathbb P^t}
\left[
e^{-r(T-t)}g(X_T^{\alpha,t,x})
\right]
&\leq
\mathbb E^{\mathbb P^t}\Bigg[
e^{-r(T-t)}g(X_T^{\alpha,t,x})
\\
&\qquad+
\int_t^T e^{-r(s-t)}
\Big(
H^\star[\Gamma(s,X_s^{\alpha,t,x})]
-
\operatorname{Tr}\big[
\Gamma(s,X_s^{\alpha,t,x})
\alpha_s\alpha_s^\top
\big]
\Big)\,ds
\Bigg].
\end{align*}
Taking the supremum over $\alpha$ gives
\begin{equation}\label{eq:dual-upper-bound}
V(t,x)\leq\mathcal D(\Gamma)(t,x),
\end{equation}
and hence
\[
V(t,x)
\leq
\inf_{\Gamma\in
\mathcal C^0_{\mathrm{loc,pol}}
([0,T)\times (0,+\infty)^d;\mathbb S^d)}
\mathcal D(\Gamma)(t,x).
\]

It remains to evaluate the dual functional at $\Gamma_V$.
The case $t=T$ is immediate, so fix $t<T$,
$\alpha\in\mathcal A^t$, and $\tau\in(t,T)$.
Since
\[
V\in
\mathcal C^{1,2}_{\mathrm{loc,pol}}
\big([0,T)\times (0,+\infty)^d;\mathbb R\big),
\]
It\^o's formula may be applied to
\[
s\mapsto
e^{-r(s-t)}
V(s,X_s^{\alpha,t,x})
\]
on $[t,\tau]$. The polynomial growth of $\nabla_xV$ on
$[0,\tau]\times (0,+\infty)^d$, together with the moment estimates above and the
boundedness of $A$, implies that the stochastic integral is a true
martingale. Using~\eqref{eq:BSB-multi}, we obtain
\begin{align}
V(t,x)
=
\mathbb E^{\mathbb P^t}\Bigg[
&e^{-r(\tau-t)}
V(\tau,X_\tau^{\alpha,t,x})
\nonumber\\
&+
\int_t^\tau e^{-r(s-t)}
\Big(
H^\star[\Gamma_V(s,X_s^{\alpha,t,x})]
-
\operatorname{Tr}\big[
\Gamma_V(s,X_s^{\alpha,t,x})
\alpha_s\alpha_s^\top
\big]
\Big)\,ds
\Bigg].
\label{eq:dual-identity-before-terminal}
\end{align}

We now let $\tau\uparrow T$. By continuity of $V$ on
$[0,T]\times (0,+\infty)^d$, the terminal condition $V(T,\cdot)=g$, and the
continuity of $X^{\alpha,t,x}$,
\[
e^{-r(\tau-t)}
V(\tau,X_\tau^{\alpha,t,x})
\longrightarrow
e^{-r(T-t)}
g(X_T^{\alpha,t,x})
\qquad
\mathbb P^t\text{-a.s.}
\]
Moreover, the polynomial growth of $V$ and the moment estimates for
$X^{\alpha,t,x}$ imply that the family of terminal random variables above
is uniformly integrable. Their expectations therefore converge.

The integrand in~\eqref{eq:dual-identity-before-terminal} is nonnegative
by~\eqref{eq:duality-gap-nonnegative}. Hence, by monotone convergence,
\begin{align*}
V(t,x)
=
\mathbb E^{\mathbb P^t}\Bigg[
&e^{-r(T-t)}
g(X_T^{\alpha,t,x})
\\
&+
\int_t^T e^{-r(s-t)}
\Big(
H^\star[\Gamma_V(s,X_s^{\alpha,t,x})]
-
\operatorname{Tr}\big[
\Gamma_V(s,X_s^{\alpha,t,x})
\alpha_s\alpha_s^\top
\big]
\Big)\,ds
\Bigg].
\end{align*}
Since this identity holds for every $\alpha\in\mathcal A^t$, taking the
supremum over $\alpha$ yields
\[
\mathcal D(\Gamma_V)(t,x)=V(t,x).
\]
Together with~\eqref{eq:dual-upper-bound}, this proves
\eqref{eq:dual-representation}.
\end{proof}

\subsection{Primal-dual bounds induced by a smooth candidate}
\label{subsec:candidate-primal-dual}

Let
\begin{equation}
\label{eq:smooth-candidate-class}
v
\in
\mathcal C^0_{\mathrm{pol}}
\big([0,T]\times(0,+\infty)^d;\mathbb R\big)
\cap
\mathcal C^{1,2}_{\mathrm{loc,pol}}
\big([0,T)\times(0,+\infty)^d;\mathbb R\big)
\end{equation}
be a smooth candidate. Its associated field $\Gamma_v$ belongs to
\[
\mathcal C^0_{\mathrm{loc,pol}}
\big([0,T)\times(0,+\infty)^d;\mathbb S^d\big).
\]
We define its BSB residual by
\begin{equation}
\label{eq:uvm-residual}
\mathcal R(v)(t,x)
\coloneqq
\partial_t v(t,x)
+
r\big(x^\top\nabla_xv(t,x)-v(t,x)\big)
+
H^\star[\Gamma_v(t,x)],
\qquad
(t,x)\in[0,T)\times(0,+\infty)^d.
\end{equation}
By~\eqref{eq:smooth-candidate-class} and the compactness of $A$,
$\mathcal R(v)$ is continuous and has polynomial growth on every strip
$[0,\tau]\times(0,+\infty)^d$, $\tau<T$.

For the residual representation below, we restrict attention to candidates
satisfying the following terminal-growth condition: there exist constants
$C>0$ and $p\geq0$ such that
\begin{equation}
\label{eq:candidate-terminal-growth}
|\partial_t v(t,x)|
+
\|\Gamma_v(t,x)\|
\leq
\frac{C(1+|x|^p)}{\sqrt{T-t}},
\qquad
(t,x)\in[0,T)\times(0,+\infty)^d.
\end{equation}
This condition is imposed as a sufficient admissibility condition ensuring
the integrability needed to pass to the terminal time. We denote by
\[
\mathfrak V
\coloneqq
\left\{
v\in
\mathcal C^0_{\mathrm{pol}}
\bigl([0,T]\times(0,+\infty)^d;\mathbb R\bigr)
\cap
\mathcal C^{1,2}_{\mathrm{loc,pol}}
\bigl([0,T)\times(0,+\infty)^d;\mathbb R\bigr)
:
v\ \text{satisfies \eqref{eq:candidate-terminal-growth}}
\right\}
\]
the corresponding class of admissible smooth candidates.

\begin{lemma}
\label{lem:candidate-residual-integrability}
Let $v\in\mathfrak V$. Then there exist constants
$C>0$ and $q\geq1$ such that
\begin{equation}
\label{eq:candidate-residual-growth}
|\mathcal R(v)(t,x)|
\leq
\frac{C(1+|x|^q)}{\sqrt{T-t}},
\qquad
(t,x)\in[0,T)\times(0,+\infty)^d.
\end{equation}
Consequently, for every $q'\geq1$ and
$(t,x)\in[0,T)\times(0,+\infty)^d$,
\begin{equation}
\label{eq:candidate-residual-integrability}
\sup_{\alpha\in\mathcal A^t}
\mathbb E^{\mathbb P^t}
\left[
\left(
\int_t^T
e^{-r(s-t)}
\left|
\mathcal R(v)(s,X_s^{\alpha,t,x})
\right|\,ds
\right)^{q'}
\right]
<+\infty.
\end{equation}
\end{lemma}

The proof is given in Appendix~\ref{app:technical}.

We now associate with $v$ the a posteriori correction
\begin{equation}
\label{eq:aposteriori-error-term}
\varepsilon(v)(t,x)
\coloneqq
\sup_{\alpha\in\mathcal A^t}
\mathbb E^{\mathbb P^t}\Bigg[
e^{-r(T-t)}
\big(
g(X_T^{\alpha,t,x})
-
v(T,X_T^{\alpha,t,x})
\big)
+
\int_t^T
e^{-r(s-t)}
\mathcal R(v)(s,X_s^{\alpha,t,x})\,ds
\Bigg].
\end{equation}
This quantity is well defined and finite. Indeed, the terminal term is
uniformly integrable over $\alpha\in\mathcal A^t$ by the polynomial
growth of $g$ and $v$ and~\Cref{lem:state-estimates}, while the running
term is absolutely integrable by
\eqref{eq:candidate-residual-integrability}.

\begin{proposition}[Candidate dual bound and residual identity]
\label{prop:candidate-dual-bound}
For every $v\in\mathfrak V$ and
$(t,x)\in[0,T]\times(0,+\infty)^d$,
\begin{equation}
\label{eq:dual-residual-identity}
\mathcal D(\Gamma_v)(t,x)
=
v(t,x)+\varepsilon(v)(t,x).
\end{equation}
Consequently,
\begin{equation}
\label{eq:uvm-duality-bound}
V(t,x)
\leq
v(t,x)+\varepsilon(v)(t,x)
=
\mathcal D(\Gamma_v)(t,x).
\end{equation}
\end{proposition}

\begin{proof}
The case $t=T$ follows immediately from the definitions. Fix
$t<T$, $\alpha\in\mathcal A^t$, and $\tau\in(t,T)$. Applying It\^o's
formula to
\[
s\mapsto
e^{-r(s-t)}
v(s,X_s^{\alpha,t,x})
\]
on $[t,\tau]$ gives
\begin{align}
&\mathbb E^{\mathbb P^t}\Bigg[
e^{-r(\tau-t)}
v(\tau,X_\tau^{\alpha,t,x})
\nonumber\\
&\qquad\qquad+
\int_t^\tau
e^{-r(s-t)}
\Big(
H^\star[\Gamma_v(s,X_s^{\alpha,t,x})]
-
\operatorname{Tr}\big[
\Gamma_v(s,X_s^{\alpha,t,x})
\alpha_s\alpha_s^\top
\big]
\Big)\,ds
\Bigg]
\nonumber\\
&\qquad=
v(t,x)
+
\mathbb E^{\mathbb P^t}
\left[
\int_t^\tau
e^{-r(s-t)}
\mathcal R(v)(s,X_s^{\alpha,t,x})\,ds
\right].
\label{eq:candidate-ito-before-terminal}
\end{align}
As in the preceding proof, the local polynomial growth of the derivatives
of $v$ and the moment estimates for the state process ensure that the
stochastic integral has zero expectation.

Letting $\tau\uparrow T$, the continuity and polynomial growth of $v$
give
\[
e^{-r(\tau-t)}
v(\tau,X_\tau^{\alpha,t,x})
\longrightarrow
e^{-r(T-t)}
v(T,X_T^{\alpha,t,x})
\qquad\text{in }L^1(\mathbb P^t).
\]
The duality-gap integral on the left-hand side of
\eqref{eq:candidate-ito-before-terminal} converges by monotone convergence,
whereas~\eqref{eq:candidate-residual-integrability} permits passage to the
limit in the signed residual integral. We obtain
\begin{align*}
&\mathbb E^{\mathbb P^t}\Bigg[
e^{-r(T-t)}
v(T,X_T^{\alpha,t,x})
+
\int_t^T
e^{-r(s-t)}
\Big(
H^\star[\Gamma_v(s,X_s^{\alpha,t,x})]
-
\operatorname{Tr}\big[
\Gamma_v(s,X_s^{\alpha,t,x})
\alpha_s\alpha_s^\top
\big]
\Big)\,ds
\Bigg]
\\
&\qquad=
v(t,x)
+
\mathbb E^{\mathbb P^t}
\left[
\int_t^T
e^{-r(s-t)}
\mathcal R(v)(s,X_s^{\alpha,t,x})\,ds
\right].
\end{align*}
Adding
\[
\mathbb E^{\mathbb P^t}
\left[
e^{-r(T-t)}
\big(
g(X_T^{\alpha,t,x})
-
v(T,X_T^{\alpha,t,x})
\big)
\right]
\]
to both sides gives, for every $\alpha\in\mathcal A^t$,
\begin{align*}
&\mathbb E^{\mathbb P^t}\Bigg[
e^{-r(T-t)}
g(X_T^{\alpha,t,x})
+
\int_t^T
e^{-r(s-t)}
\Big(
H^\star[\Gamma_v(s,X_s^{\alpha,t,x})]
-
\operatorname{Tr}\big[
\Gamma_v(s,X_s^{\alpha,t,x})
\alpha_s\alpha_s^\top
\big]
\Big)\,ds
\Bigg]
\\
&\qquad=
v(t,x)
+
\mathbb E^{\mathbb P^t}\Bigg[
e^{-r(T-t)}
\big(
g(X_T^{\alpha,t,x})
-
v(T,X_T^{\alpha,t,x})
\big)
+
\int_t^T
e^{-r(s-t)}
\mathcal R(v)(s,X_s^{\alpha,t,x})\,ds
\Bigg].
\end{align*}
Taking the supremum over $\alpha\in\mathcal A^t$ yields
\eqref{eq:dual-residual-identity}. By \Cref{prop:uvm-convex-duality},
\[
V(t,x)
\leq
\mathcal D(\Gamma_v)(t,x)
=
v(t,x)+\varepsilon(v)(t,x),
\]
which proves \eqref{eq:uvm-duality-bound}.
\end{proof}

\begin{remark} Assuming $v(T,\cdot)=g(\cdot)$ and taking the supremum inside the integrals in \eqref{eq:aposteriori-error-term} would lead to the dual formulation obtained in
  \cite[Theorem 3.2]{bank2026duality} (the one using $V_2^h$). Let us note however that this dual bound is theoretical and no numerical approximation is proposed in \cite{bank2026duality}. This is a main difference with the dual bounds we exhibit in this work, see next sections. 
\end{remark}

The preceding proposition associates with any admissible smooth candidate
$v\in\mathfrak V$ a dual upper estimate. The same candidate also induces a natural primal lower
estimate. Since $A$ is compact and
\[
a\mapsto
\operatorname{Tr}[\Gamma_v(t,x)aa^\top]
\]
is continuous, the set of maximizers is nonempty for every
$(t,x)\in[0,T)\times(0,+\infty)^d$. By the measurable maximum theorem,
there exists a Borel measurable selector
\begin{equation}
\label{eq:candidate-greedy-selector}
a^v(t,x)
\in
\arg\max_{a\in A}
\operatorname{Tr}[\Gamma_v(t,x)aa^\top].
\end{equation}
By construction,
\begin{equation}
\label{eq:candidate-zero-duality-gap}
H^\star[\Gamma_v(t,x)]
=
\operatorname{Tr}\big[
\Gamma_v(t,x)
a^v(t,x)a^v(t,x)^\top
\big].
\end{equation}
We extend $a^v$ arbitrarily to $\{T\}\times(0,+\infty)^d$ with values
in $A$, this choice has no effect on the dynamics below.

For $(t,x)\in[0,T)\times(0,+\infty)^d$, consider the closed-loop equation
\begin{equation}
\label{eq:candidate-closed-loop}
\begin{cases}
dX_s^{v,t,x}
=
rX_s^{v,t,x}\,ds
+
\operatorname{diag}(X_s^{v,t,x})
a^v(s,X_s^{v,t,x})\,dW_s,
& s\in[t,T],\\
X_t^{v,t,x}=x.
\end{cases}
\end{equation}
By compactness of $A$ and~\Hyp{hyp:ellipticity}, standard weak-existence
results for measurable uniformly nondegenerate diffusions, applied after
the logarithmic transformation, ensure that
\eqref{eq:candidate-closed-loop} admits a weak solution, see
\cite[Chapter~2, Section~6]{Kry80}. 

For any weak solution of \eqref{eq:candidate-closed-loop}, denote the
corresponding state process by $X^{v,t,x}$, its driving Brownian motion by
$W^{v,t,x}$, and the corresponding expectation by
$\mathbb E^{v,t,x}$. Set
\[
\alpha_s^{v,t,x}
\coloneqq
a^v(s,X_s^{v,t,x}),
\qquad
s\in[t,T].
\]
Since $a^v$ is Borel measurable and $X^{v,t,x}$ is continuous and
adapted, $\alpha^{v,t,x}$ is progressively measurable and $A$-valued.
Hence $\alpha^{v,t,x}$ is admissible in the weak formulation of the UVM,
see~\cite{GV02}.

\begin{proposition}[Candidate-induced primal-dual bounds]
\label{prop:candidate-primal-dual}
Let $v\in\mathfrak V$. Then, for every
$(t,x)\in[0,T)\times(0,+\infty)^d$ and every weak solution
$X^{v,t,x}$ of \eqref{eq:candidate-closed-loop},
\begin{equation}
\label{eq:candidate-primal-dual-bracket}
\mathbb E^{v,t,x}
\left[
e^{-r(T-t)}g(X_T^{v,t,x})
\right]
\leq
V(t,x)
\leq
v(t,x)+\varepsilon(v)(t,x)
=
\mathcal D(\Gamma_v)(t,x).
\end{equation}
Moreover, the corresponding primal payoff admits the representation
\begin{align}
\mathbb E^{v,t,x}
\left[
e^{-r(T-t)}g(X_T^{v,t,x})
\right]
=
v(t,x)
+
\mathbb E^{v,t,x}\Bigg[
&e^{-r(T-t)}
\big(
g(X_T^{v,t,x})
-
v(T,X_T^{v,t,x})
\big)
\nonumber\\
&+
\int_t^T
e^{-r(s-t)}
\mathcal R(v)(s,X_s^{v,t,x})\,ds
\Bigg].
\label{eq:candidate-primal-identity}
\end{align}
\end{proposition}

\begin{proof}
Fix $v\in\mathfrak V$,
$(t,x)\in[0,T)\times(0,+\infty)^d$, and a weak solution
$X^{v,t,x}$ of \eqref{eq:candidate-closed-loop}. Since
$\alpha^{v,t,x}$ is admissible in the weak formulation of the UVM, the
definition of the weak UVM value and the equivalence between the weak
formulation and the canonical strong formulation, see
\cite[Proposition~2.1]{Abb+26}, give
\[
\mathbb E^{v,t,x}
\left[
e^{-r(T-t)}g(X_T^{v,t,x})
\right]
\leq
V(t,x).
\]
This proves the primal lower bound in
\eqref{eq:candidate-primal-dual-bracket}. The upper bound and the identity
\[
\mathcal D(\Gamma_v)(t,x)
=
v(t,x)+\varepsilon(v)(t,x)
\]
follow from \Cref{prop:candidate-dual-bound}.

It remains to prove \eqref{eq:candidate-primal-identity}. Fix
$\tau\in(t,T)$. Applying It\^o's formula to
\[
s\mapsto
e^{-r(s-t)}v(s,X_s^{v,t,x}),
\qquad
s\in[t,\tau],
\]
and using \eqref{eq:candidate-zero-duality-gap}, yields
\begin{align*}
e^{-r(\tau-t)}v(\tau,X_\tau^{v,t,x})
=
v(t,x)
&+
\int_t^\tau
e^{-r(s-t)}
\mathcal R(v)(s,X_s^{v,t,x})\,ds
\\
&+
\int_t^\tau
e^{-r(s-t)}
\nabla_xv(s,X_s^{v,t,x})^\top
\operatorname{diag}(X_s^{v,t,x})
a^v(s,X_s^{v,t,x})
\,dW_s^{v,t,x}.
\end{align*}
On $[t,\tau]$, the polynomial-growth assumptions on $v$, together with
the standard moment estimates for the state process, imply that the
stochastic integral has zero expectation. Hence
\begin{equation}
\label{eq:candidate-weak-ito-expectation}
\mathbb E^{v,t,x}
\left[
e^{-r(\tau-t)}v(\tau,X_\tau^{v,t,x})
\right]
=
v(t,x)
+
\mathbb E^{v,t,x}
\left[
\int_t^\tau
e^{-r(s-t)}
\mathcal R(v)(s,X_s^{v,t,x})\,ds
\right].
\end{equation}

The moment estimates of \Cref{lem:state-estimates}, with the same proof,
hold for the weak solution considered here. Therefore, by the continuity
and polynomial growth of $v$,
\[
e^{-r(\tau-t)}v(\tau,X_\tau^{v,t,x})
\longrightarrow
e^{-r(T-t)}v(T,X_T^{v,t,x})
\]
in $L^1$ as $\tau\uparrow T$. Moreover,
\eqref{eq:candidate-residual-growth} and the same moment estimates imply
\[
\mathbb E^{v,t,x}
\left[
\int_t^T
e^{-r(s-t)}
\left|
\mathcal R(v)(s,X_s^{v,t,x})
\right|\,ds
\right]
<+\infty.
\]
Letting $\tau\uparrow T$ in
\eqref{eq:candidate-weak-ito-expectation} consequently gives
\[
\mathbb E^{v,t,x}
\left[
e^{-r(T-t)}v(T,X_T^{v,t,x})
\right]
=
v(t,x)
+
\mathbb E^{v,t,x}
\left[
\int_t^T
e^{-r(s-t)}
\mathcal R(v)(s,X_s^{v,t,x})\,ds
\right].
\]
Adding
\[
\mathbb E^{v,t,x}
\left[
e^{-r(T-t)}
\big(
g(X_T^{v,t,x})
-
v(T,X_T^{v,t,x})
\big)
\right]
\]
to both sides proves \eqref{eq:candidate-primal-identity}.
\end{proof}

\section{Discrete-time primal-dual bounds}
\label{sec:numerical-bounds}
\label{subsec:discrete-bounds}

The bounds of \Cref{sec:aposteriori-error-control} are formulated in
continuous time and involve solving stochastic control problems. In numerical
applications, these control problems are replaced by their discrete-time
counterparts. The purpose of this section is to make this passage explicit
and to identify the resulting discrete-time lower and upper bounds. We first
introduce the discrete primal problem, then derive the corresponding dual
constructions for smooth space-time and time-grid candidates, and finally
discuss their numerical evaluation by resimulation. The concrete approximation
methods and numerical experiments are presented in
\Cref{sec:numerical-methodology-experiments}.

For the numerical analysis it is convenient to work in logarithmic
coordinates. For $y\in\mathbb R^d$, the exponential $e^y$ is understood
componentwise, and we set 
\[
W(t,y)
\coloneqq
V(t,e^y).
\]
If
\[
Y_s^{\alpha,t,y}
\coloneqq
\log X_s^{\alpha,t,e^y},
\]
then It\^o's formula gives
\begin{equation}
\label{eq:log-state-dynamics}
dY_s^{\alpha,t,y}
=
b(\alpha_s)\,ds
+
\alpha_s\,dW_s^t,
\qquad
Y_t^{\alpha,t,y}=y,
\end{equation}
where $b(a)
\coloneqq
r\mathbf 1_d
-
\frac12\bigl((aa^\top)^{ii}\bigr)_{i=1}^d$. Consequently,
\begin{equation}
\label{eq:log-primal-value}
W(t,y)
=
\sup_{\alpha\in\mathcal A^t}
\mathbb E^{\mathbb P^t}
\left[
e^{-r(T-t)}
g\bigl(e^{Y_T^{\alpha,t,y}}\bigr)
\right].
\end{equation}

We impose the following additional assumption for the quantitative
discretization results of this section.

\begin{hyp}
\item 
\label{hyp:log-payoff-lipschitz}
The map $y\mapsto g(e^y)$ is globally Lipschitz on $\mathbb R^d$.
\end{hyp}

Fix $N\geq1$ and let
\[
\{0\eqqcolon t_0<t_1<\dots<t_{N-1}<t_N\coloneqq T\},
\qquad
t_n=n\tfrac{T}{N}.
\]
For $0\leq n\leq N-1$, let
$\mathcal A_N^{t_n}\subset\mathcal A^{t_n}$ denote the set of controls of
the form
\[
\alpha_s=a_k,
\qquad
s\in[t_k,t_{k+1}),
\qquad
k=n,\ldots,N-1,
\]
where $a_k$ is $A$-valued and $\mathcal F_{t_k}^{t_n}$-measurable. Since $A$ is compact, the drift $b$ and diffusion coefficient in
\eqref{eq:log-state-dynamics} are bounded, uniformly in $a\in A$, and are
independent of time and space. Thus they satisfy
\cite[Assumptions~(H1)-(H2)]{JPR19}. Moreover, conditionally on
$Y_{t_k}=y$ and $a_k=a$, the one-step transition is given exactly by
\[
Y_{t_{k+1}}
=
y
+
b(a)\tfrac{T}{N}
+
a\sqrt{\tfrac{T}{N}}\,\xi_k,
\]
where $\xi_k$ is a standard $d$-dimensional Gaussian random variable.
Hence restricting the control to the time grid introduces no
discretization of the state dynamics.

\subsection{The primal problem}

Define
\[
\mathcal W_n^\star(y)
\coloneqq
\sup_{\alpha\in\mathcal A_N^{t_n}}
\mathbb E^{\mathbb P^{t_n}}
\left[
e^{-r(T-t_n)}
g\bigl(e^{Y_T^{\alpha,t_n,y}}\bigr)
\right].
\] 

\begin{remark} The discrete value function $\mathcal W_n^\star$ is simply the logarithmic-coordinate version of the discrete UVM value function $\mathcal V_n^\star$ considered in~\cite{Abb+26}, namely $\mathcal W_n^\star(y)=\mathcal V_n^\star(e^y)$. \end{remark}

Under \Hyp{hyp:log-payoff-lipschitz}, the terminal reward in
\eqref{eq:log-primal-value} is globally Lipschitz. Since the coefficients
of \eqref{eq:log-state-dynamics} are bounded,
\cite[Theorem~2.1]{JPR19} applies directly. The deterministic discount
factor can be absorbed into the terminal reward and does not affect the
assumptions of that theorem. Hence there exists a constant $C>0$,
independent of $N$, $n$, and $y$, such that
\begin{equation}
\label{eq:discrete-log-primal-bound}
0
\leq
W(t_n,y)-\mathcal W_n^\star(y)
\leq
C\bigl(\tfrac{T}{N}\bigr)^{1/4},
\qquad
y\in\mathbb R^d.
\end{equation}
Thus $\mathcal W_n^\star\leq W(t_n,\cdot)$ is a genuine lower bound, and
\eqref{eq:discrete-log-primal-bound} quantifies the error induced by
restricting the controls to $\mathcal A_N^{t_n}$.

\subsection{Smooth space-time candidates and discrete dual bounds}
\label{subsubsec:smooth-space-time-candidates}

We next consider continuous-time candidates parameterized directly in
logarithmic coordinates. Let
\begin{equation}
\label{eq:log-numerical-candidate}
w
\in
\mathcal C_b^{2,3}
\bigl([0,T]\times\mathbb R^d;\mathbb R\bigr).
\end{equation}

For $(t,y)\in[0,T]\times\mathbb R^d$, define the log-coordinate cash gamma
by
\[
\Gamma_w^{\log}(t,y)
\coloneqq
\frac12
\Big(
\nabla_{yy}^2w(t,y)
-
\operatorname{diag}\bigl(\nabla_yw(t,y)\bigr)
\Big).
\]

For $(t,y,a)\in[0,T]\times\mathbb R^d\times A$, set
\[
F(w)(t,y,a)
\coloneqq
H^\star[\Gamma_w^{\log}(t,y)]
-
\operatorname{Tr}
\left[
\Gamma_w^{\log}(t,y)aa^\top
\right],
\]
and define the log-coordinate residual by
\[
\mathcal R^{\log}(w)(t,y)
\coloneqq
\partial_tw(t,y)
+
r\bigl(
\mathbf 1_d^\top\nabla_yw(t,y)-w(t,y)
\bigr)
+
H^\star[\Gamma_w^{\log}(t,y)].
\]

The corresponding continuous-time dual control problem is
\begin{equation}
\mathcal D^{\log}(w)(t,y)
\coloneqq
\sup_{\alpha\in\mathcal A^t}
\mathbb E^{\mathbb P^t}\Bigg[
e^{-r(T-t)}
g\bigl(e^{Y_T^{\alpha,t,y}}\bigr)
+
\int_t^T
e^{-r(s-t)}
F(w)
\bigl(
s,Y_s^{\alpha,t,y},\alpha_s
\bigr)\,ds
\Bigg],
\label{eq:log-direct-dual}
\end{equation}
while
\begin{align}
\varepsilon^{\log}(w)(t,y)
\coloneqq
\sup_{\alpha\in\mathcal A^t}
\mathbb E^{\mathbb P^t}\Bigg[
&e^{-r(T-t)}
\left(
g\bigl(e^{Y_T^{\alpha,t,y}}\bigr)
-
w\bigl(T,Y_T^{\alpha,t,y}\bigr)
\right)
\nonumber\\
&+
\int_t^T
e^{-r(s-t)}
\mathcal R^{\log}(w)
\bigl(
s,Y_s^{\alpha,t,y}
\bigr)\,ds
\Bigg].
\label{eq:log-residual-control}
\end{align}

Since $F(w)\geq0$ by definition of $H^\star$, and since It\^o's formula
applied to $e^{-r(s-t)}w(s,Y_s^{\alpha,t,y})$ yields
$\mathcal D^{\log}(w)=w+\varepsilon^{\log}(w)$, we have
\begin{equation}
\label{eq:log-continuous-upper-bounds}
W(t,y)
\leq
\mathcal D^{\log}(w)(t,y)
=
w(t,y)+\varepsilon^{\log}(w)(t,y).
\end{equation}

We next verify that the two control problems satisfy the assumptions of
\cite[Theorem~2.1]{JPR19}. Let \Hyp{hyp:log-payoff-lipschitz} hold and
let $w$ satisfy~\eqref{eq:log-numerical-candidate}. Since $A$ is compact,
\[
|H^\star[\gamma]-H^\star[\widetilde\gamma]|
\leq
C\|\gamma-\widetilde\gamma\|,
\qquad
\gamma,\widetilde\gamma\in\mathbb S^d.
\]
Moreover, $\Gamma_w^{\log}$, its spatial gradient, and its time derivative
are bounded. Consequently, $F(w)$ and $\mathcal R^{\log}(w)$ are bounded,
and there exists $C>0$ such that, for all $t,s\in[0,T]$,
$y,z\in\mathbb R^d$, and $a\in A$,
\[
|F(w)(t,y,a)-F(w)(s,z,a)|
+
|\mathcal R^{\log}(w)(t,y)-\mathcal R^{\log}(w)(s,z)|
\leq
C\bigl(
|y-z|+|t-s|^{1/2}
\bigr).
\]
Furthermore, by \Hyp{hyp:log-payoff-lipschitz} and the boundedness of
$\nabla_yw(T,\cdot)$, the map $y\mapsto g(e^y)-w(T,y)$ is globally Lipschitz on $\mathbb R^d$.

We can now define the fully discrete versions of
\eqref{eq:log-direct-dual} and \eqref{eq:log-residual-control}. Set
\begin{equation}
\label{eq:fully-discrete-log-direct-dual}
\mathcal D_N^{\log}(w)(t_n,y)
\coloneqq
\sup_{\alpha\in\mathcal A_N^{t_n}}
\mathbb E^{\mathbb P^{t_n}}\Bigg[
e^{-r(T-t_n)}
g\bigl(e^{Y_T^{\alpha,t_n,y}}\bigr)
+
\frac{T}{N}
\sum_{k=n}^{N-1}
e^{-r(t_k-t_n)}
F(w)
\bigl(
t_k,Y_{t_k}^{\alpha,t_n,y},a_k
\bigr)
\Bigg],
\end{equation}
and
\begin{align}
\varepsilon_N^{\log}(w)(t_n,y)
\coloneqq
\sup_{\alpha\in\mathcal A_N^{t_n}}
\mathbb E^{\mathbb P^{t_n}}\Bigg[
&e^{-r(T-t_n)}
\left(
g\bigl(e^{Y_T^{\alpha,t_n,y}}\bigr)
-
w\bigl(T,Y_T^{\alpha,t_n,y}\bigr)
\right)
\nonumber\\
&+
\frac{T}{N}
\sum_{k=n}^{N-1}
e^{-r(t_k-t_n)}
\mathcal R^{\log}(w)
\bigl(
t_k,Y_{t_k}^{\alpha,t_n,y}
\bigr)
\Bigg].
\label{eq:fully-discrete-log-residual}
\end{align}

There is a useful distinction between the two fully discrete quantities
above. Since $F(w)\geq0$, the discrete direct-dual problem already
dominates the discrete primal value:
\[
\mathcal W_n^\star(y)
\leq
\mathcal D_N^{\log}(w)(t_n,y).
\]
For the residual formulation, an additional error arises from replacing
the time integral of the residual by its left-point approximation. The
regularity established above yields an error of order $(T/N)^{1/2}$.
Combining these discrete-level estimates with
\eqref{eq:discrete-log-primal-bound} gives the following result.

\begin{proposition}
\label{prop:log-continuous-candidate-discrete-upper-bounds}
Let \Hyp{hyp:log-payoff-lipschitz} hold and let $w$ satisfy
\eqref{eq:log-numerical-candidate}. Then there exists a constant $C>0$,
depending on the model parameters and on the fixed candidate $w$, but
independent of $N$, $n$, and $y$, such that, for
$n=0,\ldots,N-1$ and $y\in\mathbb R^d$,
\begin{align}
\mathcal W_n^\star(y)
&\leq
w(t_n,y)
+
\varepsilon_N^{\log}(w)(t_n,y)
+
C\bigl(\tfrac{T}{N}\bigr)^{1/2},
\label{eq:fully-discrete-log-residual-discrete-upper-bound}
\\
W(t_n,y)
&\leq
\mathcal D_N^{\log}(w)(t_n,y)
+
C\bigl(\tfrac{T}{N}\bigr)^{1/4},
\label{eq:fully-discrete-log-direct-upper-bound}
\\
W(t_n,y)
&\leq
w(t_n,y)
+
\varepsilon_N^{\log}(w)(t_n,y)
+
C\bigl(\tfrac{T}{N}\bigr)^{1/4}.
\label{eq:fully-discrete-log-residual-upper-bound}
\end{align}
\end{proposition}

\begin{proof}
Since $F(w)\geq0$, the definition
\eqref{eq:fully-discrete-log-direct-dual} immediately gives
\[
\mathcal W_n^\star(y)
\leq
\mathcal D_N^{\log}(w)(t_n,y).
\]
Together with \eqref{eq:discrete-log-primal-bound}, this proves
\eqref{eq:fully-discrete-log-direct-upper-bound}.

We next consider the residual formulation. Fix
$\alpha\in\mathcal A_N^{t_n}$ and write
$Y_s=Y_s^{\alpha,t_n,y}$. Applying It\^o's formula to
$e^{-r(s-t_n)}w(s,Y_s)$ gives
\begin{align*}
&
\mathbb E^{\mathbb P^{t_n}}
\left[
e^{-r(T-t_n)}g(e^{Y_T})
\right]
\\
&\quad=
w(t_n,y)
+
\mathbb E^{\mathbb P^{t_n}}
\Bigg[
e^{-r(T-t_n)}
\bigl(
g(e^{Y_T})-w(T,Y_T)
\bigr)
\\
&\hspace{24mm}
+
\int_{t_n}^T
e^{-r(s-t_n)}
\left(
\mathcal R^{\log}(w)(s,Y_s)
-
F(w)(s,Y_s,\alpha_s)
\right)\,ds
\Bigg].
\end{align*}
Since $F(w)\geq0$,
\begin{align*}
&
\mathbb E^{\mathbb P^{t_n}}
\left[
e^{-r(T-t_n)}g(e^{Y_T})
\right]
\\
&\quad\leq
w(t_n,y)
+
\mathbb E^{\mathbb P^{t_n}}
\Bigg[
e^{-r(T-t_n)}
\bigl(
g(e^{Y_T})-w(T,Y_T)
\bigr)
\\
&\hspace{24mm}
+
\int_{t_n}^T
e^{-r(s-t_n)}
\mathcal R^{\log}(w)(s,Y_s)\,ds
\Bigg].
\end{align*}

It remains to replace the residual integral by its left-point
approximation. Since the coefficients of
\eqref{eq:log-state-dynamics} are bounded,
\[
\mathbb E^{\mathbb P^{t_n}}
\left[
\left|
Y_s-Y_{t_k}
\right|
\right]
\leq
C|s-t_k|^{1/2},
\qquad
s\in[t_k,t_{k+1}],
\]
uniformly over
$\alpha\in\mathcal A_N^{t_n}$. Using the regularity of
$\mathcal R^{\log}(w)$ established above, together with the boundedness
of $\mathcal R^{\log}(w)$ to control the variation of the discount
factor, we obtain
\[
\sup_{\alpha\in\mathcal A_N^{t_n}}
\mathbb E^{\mathbb P^{t_n}}
\left[
\left|
\int_{t_n}^T
e^{-r(s-t_n)}
\mathcal R^{\log}(w)(s,Y_s)\,ds
-
\frac{T}{N}
\sum_{k=n}^{N-1}
e^{-r(t_k-t_n)}
\mathcal R^{\log}(w)(t_k,Y_{t_k})
\right|
\right]
\leq
C\bigl(\tfrac{T}{N}\bigr)^{1/2}.
\]
Taking the supremum over
$\alpha\in\mathcal A_N^{t_n}$ therefore yields
\eqref{eq:fully-discrete-log-residual-discrete-upper-bound}.

Finally, combining
\eqref{eq:fully-discrete-log-residual-discrete-upper-bound} with
\eqref{eq:discrete-log-primal-bound} gives
\[
W(t_n,y)
\leq
w(t_n,y)
+
\varepsilon_N^{\log}(w)(t_n,y)
+
C
\left[
\bigl(\tfrac{T}{N}\bigr)^{1/4}
+
\bigl(\tfrac{T}{N}\bigr)^{1/2}
\right].
\]
Since $T$ is fixed and $N\geq1$, the second term can be absorbed into the
first, which proves
\eqref{eq:fully-discrete-log-residual-upper-bound}.
\end{proof}

The regularity condition~\eqref{eq:log-numerical-candidate} is naturally
satisfied by the neural-network candidates considered below. In
particular, a finite feedforward network with $\tanh$ activation in the
hidden layers and an affine output layer belongs to
$\mathcal C_b^{2,3}$ for fixed network parameters. Thus, when the network
is parameterized directly in the log-price variables, the regularity
required above is automatic.

\begin{remark}
The discrete primal values are monotone under nested refinements of the
control grid. In particular, if the grid with $M$ time steps refines the
grid with $N$ time steps, then
\[
\mathcal A_N^{t_n}
\subset
\mathcal A_M^{t_m}
\]
whenever $t_n=t_m$, and therefore
\[
\mathcal W_n^\star(y)
\leq
\mathcal W_m^\star(y).
\]
For example, along dyadic refinements $N=2^{N'}$, the values at the
initial time form an increasing sequence:
\[
\mathcal W_{0,2^{N'}}^\star(y)
\leq
\mathcal W_{0,2^{N'+1}}^\star(y).
\]
Together with \eqref{eq:discrete-log-primal-bound}, this yields
\[
\mathcal W_{0,2^{N'}}^\star(y)
\uparrow
W(0,y)
\qquad\text{as }N'\to\infty.
\]

No analogous monotonicity is available in general for
$\mathcal D_N^{\log}(w)$ or $\varepsilon_N^{\log}(w)$, even for a fixed
smooth candidate $w$ and along nested refinements. Indeed, refining the
grid enlarges the admissible control class, but at the same time changes
the left-point approximation of the running integral. Since these
quadrature approximations are not monotone in general, the resulting
dual quantities need not be ordered from one grid to the next.
\end{remark}

\subsection{Time-grid candidates}
\label{subsubsec:time-grid-candidates}

The previous construction assumes that the candidate is a smooth function
of both time and space. This is not the case for the critics produced by
the SPG algorithm \cite{Abb+26}, which form a family of spatial functions indexed by the
time grid. We therefore consider a sequence
\[
(w_k)_{k=0}^N,
\qquad
w_k:\mathbb R^d\to\mathbb R,
\]
parameterized directly in logarithmic coordinates, with the convention
\[
w_N(y)=g(e^y).
\]
For $k=0,\ldots,N-1$, we assume that $w_k\in\mathcal C^2(\mathbb R^d)$,
has at most linear growth, and has bounded first and second derivatives.
No differentiability is imposed on $w_N$.

For $k=0,\ldots,N-1$, define the forward difference
\[
D_t^+w_k(y)
\coloneqq
(\tfrac{T}{N})^{-1}
\bigl(
w_{k+1}(y)-w_k(y)
\bigr),
\]
and
\[
\Gamma_k^{\log}(y)
\coloneqq
\frac12
\left(
D^2w_k(y)
-
\operatorname{diag}\bigl(\nabla w_k(y)\bigr)
\right).
\]
The residual evaluated numerically at time $t_k$ is then
\begin{equation}
\label{eq:time-grid-residual}
\widehat{\mathcal R}_k^{\log}(w)(y)
\coloneqq
D_t^+w_k(y)
+
r\bigl(
\mathbf 1_d^\top\nabla w_k(y)-w_k(y)
\bigr)
+
H^\star[\Gamma_k^{\log}(y)].
\end{equation}
Notice that $w_N=g\circ\exp$ enters
\eqref{eq:time-grid-residual} only through the value
$D_t^+w_{N-1}$ and is never differentiated. In analogy with \eqref{eq:fully-discrete-log-residual}, define
\begin{equation}
\label{eq:time-grid-residual-control}
\widehat\varepsilon_N(w)(t_n,y)
\coloneqq
\sup_{\alpha\in\mathcal A_N^{t_n}}
\mathbb E^{\mathbb P^{t_n}}
\left[
\frac{T}{N}
\sum_{k=n}^{N-1}
e^{-r(t_k-t_n)}
\widehat{\mathcal R}_k^{\log}(w)
\bigl(Y_{t_k}^{\alpha,t_n,y}\bigr)
\right].
\end{equation}

To derive an upper bound for the discrete-control value
$\mathcal W_n^\star$ in terms of the residual evaluated numerically, we
first fix $\alpha\in\mathcal A_N^{t_n}$ and relate its discounted payoff
to the time-grid candidate. Write
$Y_s=Y_s^{\alpha,t_n,y}$. Since $w_N=g\circ\exp$, the discounted payoff
can be written as
\[
\mathbb E^{\mathbb P^{t_n}}
\left[
e^{-r(T-t_n)}g(e^{Y_T})
\right]
=
\mathbb E^{\mathbb P^{t_n}}
\left[
e^{-r(T-t_n)}w_N(Y_T)
\right].
\]
Hence, in order to compare this payoff with the numerical residual, it is
natural to start from the difference
\[
\mathbb E^{\mathbb P^{t_n}}
\left[
e^{-r(T-t_n)}w_N(Y_T)
\right]
-
w_n(y).
\]

This difference can be expanded exactly by telescoping over the time
grid:
\[
e^{-r(T-t_n)}w_N(Y_T)-w_n(y)
=\sum_{k=n}^{N-1}
\left[
e^{-r(t_{k+1}-t_n)}
w_{k+1}(Y_{t_{k+1}})
-
e^{-r(t_k-t_n)}
w_k(Y_{t_k})
\right].
\]
Thus the terminal payoff is linked to the initial value $w_n(y)$ through
the successive one-step changes of the time-grid candidate along the
controlled state process.

Consider one of these one-step changes. We decompose it as
\begin{align*}
&
e^{-r(t_{k+1}-t_n)}
w_{k+1}(Y_{t_{k+1}})
-
e^{-r(t_k-t_n)}
w_k(Y_{t_k})
\\
&\quad=
e^{-r(t_{k+1}-t_n)}
\bigl(
w_{k+1}-w_k
\bigr)(Y_{t_{k+1}})
+
\left[
e^{-r(t_{k+1}-t_n)}
w_k(Y_{t_{k+1}})
-
e^{-r(t_k-t_n)}
w_k(Y_{t_k})
\right].
\end{align*}
The first term represents the change from the critic $w_k$ to the next
critic $w_{k+1}$, evaluated at the state reached at time $t_{k+1}$.
The second term represents the evolution of the state over
$[t_k,t_{k+1}]$ while the critic $w_k$ is kept fixed. The second term can be evaluated exactly by It\^o's formula. For
$a\in A$, set
\[
L_k^a w_k(y)
\coloneqq
r\bigl(
\mathbf 1_d^\top\nabla w_k(y)-w_k(y)
\bigr)
+
\operatorname{Tr}
\left[
\Gamma_k^{\log}(y)aa^\top
\right].
\]
Since the control is equal to $a_k$ on $[t_k,t_{k+1})$,
\[
\mathbb E^{\mathbb P^{t_n}}
\left[
e^{-r(t_{k+1}-t_n)}
w_k(Y_{t_{k+1}})
-
e^{-r(t_k-t_n)}
w_k(Y_{t_k})
\right]
=\mathbb E^{\mathbb P^{t_n}}
\left[
\int_{t_k}^{t_{k+1}}
e^{-r(s-t_n)}
L_k^{a_k}w_k(Y_s)\,ds
\right].
\]
After summing over $k=n,\ldots,N-1$, we therefore obtain the exact
identity
\begin{equation}
\label{eq:time-grid-exact-representation}
\begin{aligned}
&
\mathbb E^{\mathbb P^{t_n}}
\left[
e^{-r(T-t_n)}g(e^{Y_T})
\right]
-
w_n(y)
\\
&\quad=
\mathbb E^{\mathbb P^{t_n}}
\Bigg[
\sum_{k=n}^{N-1}
e^{-r(t_{k+1}-t_n)}
(w_{k+1}-w_k)(Y_{t_{k+1}})
+
\sum_{k=n}^{N-1}
\int_{t_k}^{t_{k+1}}
e^{-r(s-t_n)}
L_k^{a_k}w_k(Y_s)\,ds
\Bigg].
\end{aligned}
\end{equation}
This is the exact representation that one would like to relate to the
residual evaluated numerically. Before replacing the controlled
generator by the supremum $H^\star$, the corresponding numerical
quantity is
\[
\frac{T}{N}
\sum_{k=n}^{N-1}
e^{-r(t_k-t_n)}
\left[
D_t^+w_k(Y_{t_k})
+
L_k^{a_k}w_k(Y_{t_k})
\right].
\]
Using
\[
\frac{T}{N}D_t^+w_k
=
w_{k+1}-w_k,
\]
this becomes
\begin{equation}
\label{eq:time-grid-left-point-representation}
\sum_{k=n}^{N-1}
e^{-r(t_k-t_n)}
(w_{k+1}-w_k)(Y_{t_k})
+\frac{T}{N}
\sum_{k=n}^{N-1}
e^{-r(t_k-t_n)}
L_k^{a_k}w_k(Y_{t_k}).
\end{equation}

The comparison between
\eqref{eq:time-grid-exact-representation} and
\eqref{eq:time-grid-left-point-representation} is now explicit.
There are two differences. First, the exact telescoping identity contains
\[
e^{-r(t_{k+1}-t_n)}
(w_{k+1}-w_k)(Y_{t_{k+1}}),
\]
whereas the forward-difference residual uses
\[
e^{-r(t_k-t_n)}
(w_{k+1}-w_k)(Y_{t_k}).
\]
The exact change from $w_k$ to $w_{k+1}$ is therefore evaluated at the
right endpoint of the interval, while the numerical residual evaluates
the same change at the left endpoint. Second, the exact identity contains the generator contribution
\[
\int_{t_k}^{t_{k+1}}
e^{-r(s-t_n)}
L_k^{a_k}w_k(Y_s)\,ds,
\]
whereas the numerical residual replaces it by the left-point
approximation
\[
\frac{T}{N}
e^{-r(t_k-t_n)}
L_k^{a_k}w_k(Y_{t_k}).
\]
Hence the passage from the exact payoff identity to the quantity used
numerically consists of the two replacements
\[
e^{-r(t_k-t_n)}
(w_{k+1}-w_k)(Y_{t_k})
\quad\text{instead of}\quad
e^{-r(t_{k+1}-t_n)}
(w_{k+1}-w_k)(Y_{t_{k+1}}),
\]
and
\[
\frac{T}{N}
e^{-r(t_k-t_n)}
L_k^{a_k}w_k(Y_{t_k})
\quad\text{instead of}\quad
\int_{t_k}^{t_{k+1}}
e^{-r(s-t_n)}
L_k^{a_k}w_k(Y_s)\,ds.
\]

The role of the consistency defect~\eqref{eq:time-grid-consistency-defect} introduced below is precisely to
measure the error in this approximation step. Once this error has been
accounted for, the inequality
\[
\operatorname{Tr}
\left[
\Gamma_k^{\log}(y)a_ka_k^\top
\right]
\leq
H^\star[\Gamma_k^{\log}(y)]
\]
turns the last expression into the numerical residual
$\widehat{\mathcal R}_k^{\log}(w)$ and therefore yields the desired upper
bound.

We define the consistency defect by
\begin{align}
\mathfrak E_N(w)(t_n,y)
\coloneqq
\sup_{\alpha\in\mathcal A_N^{t_n}}
\mathbb E^{\mathbb P^{t_n}}
\Bigg[
&
\left|
\sum_{k=n}^{N-1}
\Big(
e^{-r(t_{k+1}-t_n)}
(w_{k+1}-w_k)
\bigl(Y_{t_{k+1}}^{\alpha,t_n,y}\bigr)
\right.
\nonumber\\[-1mm]
&\hspace{33mm}\left.
-
e^{-r(t_k-t_n)}
(w_{k+1}-w_k)
\bigl(Y_{t_k}^{\alpha,t_n,y}\bigr)
\Big)
\right|
\nonumber\\
&+
\left|
\sum_{k=n}^{N-1}
\int_{t_k}^{t_{k+1}}
e^{-r(s-t_n)}
L_k^{a_k}w_k
\bigl(Y_s^{\alpha,t_n,y}\bigr)\,ds
\right.
\nonumber\\[-1mm]
&\hspace{18mm}\left.
-
\frac{T}{N}
\sum_{k=n}^{N-1}
e^{-r(t_k-t_n)}
L_k^{a_k}w_k
\bigl(Y_{t_k}^{\alpha,t_n,y}\bigr)
\right|
\Bigg].
\label{eq:time-grid-consistency-defect}
\end{align}
Thus $\mathfrak E_N(w)$ is not an additional residual. It measures the
difference between the exact identity linking the terminal payoff to the
time-grid candidate and the left-point approximation of that identity
used in the numerical residual. No temporal regularity is imposed on the
sequence $(w_k)_{k=0}^N$, and we therefore do not assign a convergence
rate to this defect.

The following result gives the corresponding a posteriori bound.

\begin{proposition}
\label{prop:time-grid-candidate-upper-bound}
Let \Hyp{hyp:log-payoff-lipschitz} hold and let
$(w_k)_{k=0}^N$ satisfy the assumptions above. Then, for
$n=0,\ldots,N-1$ and $y\in\mathbb R^d$,
\begin{align}
\mathcal W_n^\star(y)
&\leq
w_n(y)
+
\widehat\varepsilon_N(w)(t_n,y)
+
\mathfrak E_N(w)(t_n,y),
\label{eq:time-grid-discrete-upper-bound}
\\
W(t_n,y)
&\leq
w_n(y)
+
\widehat\varepsilon_N(w)(t_n,y)
+
\mathfrak E_N(w)(t_n,y)
+
C\bigl(\tfrac{T}{N}\bigr)^{1/4},
\label{eq:time-grid-continuous-upper-bound}
\end{align}
where the constant $C>0$ in
\eqref{eq:time-grid-continuous-upper-bound} is the one appearing in
\eqref{eq:discrete-log-primal-bound} and is independent of the candidate
sequence and of $N$.
\end{proposition}

\begin{proof}
Fix $\alpha\in\mathcal A_N^{t_n}$ and write
$Y_s=Y_s^{\alpha,t_n,y}$. By
\eqref{eq:time-grid-exact-representation} and the definition of
$\mathfrak E_N(w)$ in \eqref{eq:time-grid-consistency-defect},
\begin{align*}
&
\mathbb E^{\mathbb P^{t_n}}
\left[
e^{-r(T-t_n)}g(e^{Y_T})
\right]
\\
&\qquad\leq
w_n(y)
+
\mathbb E^{\mathbb P^{t_n}}
\Bigg[
\sum_{k=n}^{N-1}
e^{-r(t_k-t_n)}
(w_{k+1}-w_k)(Y_{t_k})
\\
&\hspace{30mm}
+
\frac{T}{N}
\sum_{k=n}^{N-1}
e^{-r(t_k-t_n)}
L_k^{a_k}w_k(Y_{t_k})
\Bigg]
+
\mathfrak E_N(w)(t_n,y).
\end{align*}
Using
\[
w_{k+1}(y)-w_k(y)
=
\frac{T}{N}D_t^+w_k(y)
\]
and, by the definition of $H^\star$,
\[
L_k^{a_k}w_k(y)
\leq
r\bigl(
\mathbf 1_d^\top\nabla w_k(y)-w_k(y)
\bigr)
+
H^\star[\Gamma_k^{\log}(y)],
\]
we obtain
\begin{align*}
&
\mathbb E^{\mathbb P^{t_n}}
\left[
e^{-r(T-t_n)}g(e^{Y_T})
\right]
\\
&\qquad\leq
w_n(y)
+
\mathbb E^{\mathbb P^{t_n}}
\left[
\frac{T}{N}
\sum_{k=n}^{N-1}
e^{-r(t_k-t_n)}
\widehat{\mathcal R}_k^{\log}(w)(Y_{t_k})
\right]
+
\mathfrak E_N(w)(t_n,y).
\end{align*}
Taking the supremum over
$\alpha\in\mathcal A_N^{t_n}$ yields
\eqref{eq:time-grid-discrete-upper-bound} by
\eqref{eq:time-grid-residual-control}.

Finally, \eqref{eq:discrete-log-primal-bound} gives
\[
W(t_n,y)
\leq
\mathcal W_n^\star(y)
+
C\bigl(\tfrac{T}{N}\bigr)^{1/4},
\]
which, together with
\eqref{eq:time-grid-discrete-upper-bound}, proves
\eqref{eq:time-grid-continuous-upper-bound}.
\end{proof}

The proposition is a finite-grid a posteriori result: it does not require
the trained families obtained for different values of $N$ to satisfy
uniform derivative bounds, nor does it impose temporal regularity on the
critics. In particular, at the initial time, if for a sequence of grids
and associated candidate families
\[
\mathfrak E_N(w)(0,y)\longrightarrow0,
\]
then the only remaining universal discretization correction in
\eqref{eq:time-grid-continuous-upper-bound} is the
$(T/N)^{1/4}$ term inherited from
\eqref{eq:discrete-log-primal-bound}. More generally, the same conclusion
holds along any sequence $n=n(N)$ for which $t_{n(N)}$ converges to the
time at which the value is to be evaluated, provided the corresponding
consistency defect tends to zero.

For the SPG critics considered here, the spatial regularity assumptions
are natural. Indeed, a finite one-hidden-layer network with GELU
activation is smooth, has at most linear growth, and has bounded first and
second derivatives for fixed network parameters. The terminal function
$w_N=g\circ\exp$ is treated only through its values and therefore
requires no regularity beyond \Hyp{hyp:log-payoff-lipschitz}.

\subsection{Numerical evaluation by resimulation}
\label{subsubsec:resimulation}

The discrete control problems introduced above still involve a supremum
over admissible controls and cannot in general be solved exactly. In the
numerical implementation, an admissible control $\widehat\alpha$ is first
obtained and the corresponding expectation is then
estimated by an independent Monte Carlo resimulation. The control
$\widehat\alpha$ is generally suboptimal for the corresponding stochastic
control problem. It is important to distinguish the resulting
optimization error from the Monte Carlo sampling error.

For the primal problem, any
$\widehat\alpha\in\mathcal A_N^{t_n}$ satisfies
\[
\mathbb E^{\mathbb P^{t_n}}
\left[
e^{-r(T-t_n)}
g\bigl(
e^{Y_T^{\widehat\alpha,t_n,y}}
\bigr)
\right]
\leq
\mathcal W_n^\star(y)
\leq
W(t_n,y).
\]
Thus the optimization error is one-sided in the favorable direction:
evaluating a suboptimal admissible control can only decrease the primal
value. Consequently, apart from Monte Carlo sampling error, resimulation
under $\widehat\alpha$ preserves the lower-bound property.

The situation is different for the dual control problems. For the direct
dual problem, the definition
\eqref{eq:fully-discrete-log-direct-dual} gives, for every admissible
$\widehat\alpha$,
\[
\mathbb E^{\mathbb P^{t_n}}\Bigg[
e^{-r(T-t_n)}
g\bigl(
e^{Y_T^{\widehat\alpha,t_n,y}}
\bigr)
+
\frac{T}{N}
\sum_{k=n}^{N-1}
e^{-r(t_k-t_n)}
F(w)
\bigl(
t_k,Y_{t_k}^{\widehat\alpha,t_n,y},
\widehat\alpha_{t_k}
\bigr)
\Bigg]
\leq
\mathcal D_N^{\log}(w)(t_n,y).
\]
Hence replacing the supremum in
$\mathcal D_N^{\log}(w)$ by resimulation under
$\widehat\alpha$ approximates the dual value from below. Since the upper
bound established above involves $\mathcal D_N^{\log}(w)$ itself,
\[
W(t_n,y)
\leq
\mathcal D_N^{\log}(w)(t_n,y)
+
C\bigl(\tfrac{T}{N}\bigr)^{1/4},
\]
the same replacement does not in general preserve the upper-bound
property.

The same issue arises for the residual formulation. For every admissible
$\widehat\alpha$,
\[
\mathbb E^{\mathbb P^{t_n}}\Bigg[
e^{-r(T-t_n)}
\left(
g\bigl(
e^{Y_T^{\widehat\alpha,t_n,y}}
\bigr)
-
w\bigl(
T,Y_T^{\widehat\alpha,t_n,y}
\bigr)
\right)
+
\frac{T}{N}
\sum_{k=n}^{N-1}
e^{-r(t_k-t_n)}
\mathcal R^{\log}(w)
\bigl(
t_k,Y_{t_k}^{\widehat\alpha,t_n,y}
\bigr)
\Bigg]
\leq
\varepsilon_N^{\log}(w)(t_n,y).
\]
Thus a resimulated approximation of
$\varepsilon_N^{\log}(w)$ is again biased downward by the optimization
error, whereas the theoretical upper bound requires the supremum itself.

The same conclusion applies to the time-grid formulation. In particular,
evaluating the expectation in
\eqref{eq:time-grid-residual-control} under a fixed admissible control
produces a value no larger than
$\widehat\varepsilon_N(w)(t_n,y)$. Likewise, if the consistency defect
$\mathfrak E_N(w)(t_n,y)$ is itself approximated by optimization and
resimulation, evaluation under a fixed control does not in general
recover the supremum appearing in
\eqref{eq:time-grid-consistency-defect}. Therefore, replacing these suprema by values obtained under fixed resimulated controls does not preserve the theoretical upper bound.

Monte Carlo simulation introduces a separate statistical error in the
evaluation of the expectations above. On the primal side, the exact
expectation obtained under the admissible control $\widehat\alpha$ is a
lower bound for $W(t_n,y)$, so the resimulated value retains this
interpretation up to Monte Carlo sampling error. On the dual side, Monte Carlo simulation only approximates the expectation
associated with the chosen control $\widehat\alpha$. Even if this
expectation were evaluated exactly, it would be no larger than the
supremum defining the theoretical dual quantity. The loss of the
upper-bound property therefore comes from the optimization error rather
than from the Monte Carlo approximation.

\section{Numerical methodology and experiments}
\label{sec:numerical-methodology-experiments}

We now turn to the numerical construction of the candidates and to the
evaluation of the primal and dual quantities introduced above. We first
describe the backward SPG and PINN approximations and the common evaluation
procedure, and then report the five- and two-dimensional experiments.

\subsection{Numerical methodology}
\label{subsec:numerical-methodology}

\subsubsection{Backward SPG approximation and derivative-informed critics}
\label{subsubsec:spg-numerical-scheme}

We use the backward SPG architecture of~\cite{Abb+26} on the time grid
\[
0=t_0<t_1<\cdots<t_N=T,
\qquad
t_n=\frac{nT}{N},
\]
and parameterize the critics directly in logarithmic coordinates. At each
time level $t_n$, $n=N-1,\ldots,0$, a separate stochastic actor
$\pi_{\theta_n}$ and critic $\mathcal W_{\phi_n}$ are trained, while the
already learned critic at time $t_{n+1}$ is frozen and provides the
continuation value. The resulting candidate sequence is
\[
w_n
\coloneqq
\mathcal W_{\phi_n^\star},
\qquad
n=0,\ldots,N-1,
\qquad
w_N(y)
\coloneqq
g(e^y),
\]
and therefore coincides with the class of time-grid candidates considered
in~\Cref{subsubsec:time-grid-candidates}.

The actor parametrization, admissible-control map, PPO update, correlation
penalty, state-sampling distributions, and backward training procedure are
those of~\cite{Abb+26}. We recall only the critic regression and the
derivative-informed modifications relevant for the present a posteriori
analysis.

At a fixed backward step $n$ and training epoch $\ell$, let
$Y_n^{(\ell)}\sim\mu_n$ be a sampled log-state. Keeping fixed the
realization of the policy exploration noise, denote by
\[
y\mapsto a_n^{(\ell)}(y)
\]
the corresponding state-dependent feedback control. If
$\xi_n^{(\ell)}\sim\mathcal N(0,I_d)$ is the independent Gaussian
innovation used for the state transition, define
\[
\Phi_n^{(\ell)}(y)
\coloneqq
y
+
b\bigl(a_n^{(\ell)}(y)\bigr)\frac{T}{N}
+
a_n^{(\ell)}(y)\sqrt{\frac{T}{N}}\,
\xi_n^{(\ell)}.
\]
Since the control is constant over one time interval, this is the exact
one-step transition of the log-state dynamics. The corresponding
state-dependent Bellman target is
\[
\mathcal Z_n^{(\ell)}(y)
\coloneqq
e^{-rT/N}
w_{n+1}\bigl(\Phi_n^{(\ell)}(y)\bigr),
\]
and its realization at the sampled state is
\[
Z_n^{(\ell)}
\coloneqq
\mathcal Z_n^{(\ell)}\bigl(Y_n^{(\ell)}\bigr).
\]
The standard critic minimizes
\begin{equation}
\label{eq:backward-standard-critic-loss}
\mathcal L_{n,W}^{\mathrm C}(\phi)
=
\mathbb E_n^{(\ell)}
\left[
\left|
\mathcal W_\phi\bigl(Y_n^{(\ell)}\bigr)
-
Z_n^{(\ell)}
\right|^2
\right].
\end{equation}
At the end of each backward step, as in~\cite{Abb+26}, a final critic
epoch is performed under the deterministic policy induced by the trained
actor.

\paragraph{First-order Sobolev training.}

Value regression alone does not control the spatial derivatives of the
critic. This is already relevant for the martingale dual estimate used
in~\cite{Abb+26}. Indeed, the stochastic-control duality
of~\cite{HLR16}, applied directly to the log-state dynamics, leads to a
martingale correction whose discrete approximation contains terms of the
form
\[
e^{-r(t_k-t_n)}
\nabla_y w_k(Y_{t_k})^\top
a_k
\bigl(
W_{t_{k+1}}-W_{t_k}
\bigr).
\]
Thus the quality of this diagnostic depends directly on the approximation
of the log-gradient $\nabla_y w_k$.

For this reason,~\cite{Abb+26} augments the value regression with a
first-order Sobolev loss, following the Sobolev Training methodology
of~\cite{Cza+17}, see also~\cite{HSW90,HS20}. At each training sample, set
\[
\bar a_n^{(\ell)}
\coloneqq
a_n^{(\ell)}\bigl(Y_n^{(\ell)}\bigr),
\]
and define the corresponding frozen-action target map
\[
\widetilde{\mathcal Z}_n^{(\ell)}(y)
\coloneqq
e^{-rT/N}
w_{n+1}
\Big(y+ b\bigl(\bar a_n^{(\ell)}\bigr)\frac{T}{N} + \bar a_n^{(\ell)} \sqrt{\frac{T}{N}}\, \xi_n^{(\ell)} \Big).
\]
By construction,
$\widetilde{\mathcal Z}_n^{(\ell)}(Y_n^{(\ell)})=Z_n^{(\ell)}$, so this does not alter the value target. The first-order derivative target is
\[
\nabla_y
\widetilde{\mathcal Z}_n^{(\ell)}
\bigl(Y_n^{(\ell)}\bigr).
\]
Thus the actor is evaluated at the sampled state, but the resulting action
is treated as fixed when differentiating with respect to the state. In particular, no derivative is propagated through the actor. This is the convention used in~\cite{Abb+26}. We denote the corresponding loss by
$\mathcal L_{n,DW}^{\mathrm C}$ and otherwise use the construction
of~\cite{Abb+26} unchanged.

\paragraph{Second-order Sobolev training.}

The a posteriori quantities introduced in
\Cref{subsubsec:time-grid-candidates} additionally require second-order
spatial information. In logarithmic coordinates,
\[
\Gamma_n^{\log}(y)
=
\frac12
\left(
D_y^2w_n(y)
-
\operatorname{diag}\bigl(\nabla_yw_n(y)\bigr)
\right),
\]
and the numerical residual
$\widehat{\mathcal R}_n^{\log}(w)$ in
\eqref{eq:time-grid-residual} depends on this matrix through
$H^\star[\Gamma_n^{\log}]$. Consequently, accurate values and gradients
alone do not guarantee an accurate evaluation of the a posteriori
residual. We therefore extend the first-order Sobolev training
of~\cite{Abb+26} by incorporating second-order information.

A direct approach would compare the full Hessians
$D_y^2\mathcal W_\phi$ and
$D_y^2\widetilde{\mathcal Z}_n^{(\ell)}$ at every training sample. To avoid constructing and storing a matrix-valued target, we use
random directional second derivatives as in stochastic Sobolev
Training~\cite{Cza+17}. Let
\[
\zeta_n^{(\ell)}
\sim
\mathcal N(0,I_d)
\]
be independent of the other sampled quantities. For a twice
differentiable function $f:\mathbb R^d\to\mathbb R$, set
\[
D_{\zeta\zeta}^2f(y)
\coloneqq
\zeta^\top D_y^2f(y)\zeta.
\]
The scalar second-order target is
\[
Z_{n,2}^{(\ell)}
\coloneqq
D_{\zeta_n^{(\ell)}\zeta_n^{(\ell)}}^2
\widetilde{\mathcal Z}_n^{(\ell)}
\bigl(Y_n^{(\ell)}\bigr),
\]
and the corresponding critic loss is
\begin{equation}
\label{eq:backward-directional-hessian-loss}
\mathcal L_{n,D^2W}^{\mathrm C}(\phi)
\coloneqq
\mathbb E_n^{(\ell)}
\left[
\left|
D_{\zeta_n^{(\ell)}\zeta_n^{(\ell)}}^2
\mathcal W_\phi\bigl(Y_n^{(\ell)}\bigr)
-
Z_{n,2}^{(\ell)}
\right|^2
\right].
\end{equation}

Directional probing retains information about the complete Hessian,
including its mixed derivatives. Indeed, if
$B\in\mathbb S^d$ and $\zeta\sim\mathcal N(0,I_d)$, then
\[
\mathbb E_\zeta
\left[
(\zeta^\top B\zeta)^2
\right]
=
2\|B\|_F^2+\operatorname{Tr}(B)^2,
\]
and therefore
\[
2\|B\|_F^2
\leq
\mathbb E_\zeta
\left[
(\zeta^\top B\zeta)^2
\right]
\leq
(d+2)\|B\|_F^2.
\]
Hence the expected directional squared error is equivalent to the squared
Frobenius norm of the full Hessian error.

The directional second derivatives are evaluated by automatic
differentiation. For a fixed direction $\zeta$,
\[
D_\zeta f(y)
=
\nabla_y f(y)^\top\zeta,
\qquad
D_{\zeta\zeta}^2f(y)
=
\zeta^\top\nabla_y D_\zeta f(y),
\]
so that only directional derivatives, or equivalently
Hessian-vector products, are required during training. After training,
the full Hessians $D_y^2w_n$ are evaluated by automatic differentiation
at the points needed to compute $\Gamma_n^{\log}$ and
$\widehat{\mathcal R}_n^{\log}(w)$.

For $n\leq N-2$, the continuation critic $w_{n+1}$ is a smooth neural
network, and the first- and second-order targets are obtained by automatic
differentiation of the frozen-action map
$\widetilde{\mathcal Z}_n^{(\ell)}$. Thus, consistently with the
first-order construction of~\cite{Abb+26}, the sampled action is held fixed
when differentiating with respect to the state. We use the same convention
for the second-order target.

At the last backward step, $w_N=g\circ\exp$ need not be differentiable.
We therefore approximate the derivative targets by centered finite
differences. Writing $(e_i)_{1\leq i\leq d}$ for the canonical basis of
$\mathbb R^d$ and $\delta_{\mathrm{FD}}>0$ for the finite-difference step,
\[
\partial_{y_i}
\widetilde{\mathcal Z}_{N-1}^{(\ell)}(y)
\approx
\frac{
\widetilde{\mathcal Z}_{N-1}^{(\ell)}
(y+\delta_{\mathrm{FD}}e_i)
-
\widetilde{\mathcal Z}_{N-1}^{(\ell)}
(y-\delta_{\mathrm{FD}}e_i)
}{
2\delta_{\mathrm{FD}}
},
\]
while, for a sampled direction $\zeta$,
\[
D_{\zeta\zeta}^2
\widetilde{\mathcal Z}_{N-1}^{(\ell)}(y)
\approx
\frac{
\widetilde{\mathcal Z}_{N-1}^{(\ell)}
(y+\delta_{\mathrm{FD}}\zeta)
-
2\widetilde{\mathcal Z}_{N-1}^{(\ell)}(y)
+
\widetilde{\mathcal Z}_{N-1}^{(\ell)}
(y-\delta_{\mathrm{FD}}\zeta)
}{
\delta_{\mathrm{FD}}^2
}.
\]
The same realizations of the exogenous random variables and the same
frozen action $\bar a_{N-1}^{(\ell)}$ are used at all perturbed states;
the feedback control is not reevaluated at the perturbed points. Since the
perturbations are performed in logarithmic coordinates, no positivity
constraint on the perturbed state is required.

The normalization of the value and derivative losses is the same as
in~\cite{Abb+26} and is not repeated here. With this convention, the
second-order Sobolev critic minimizes
\begin{equation}
\label{eq:sobolev2-critic-loss}
\mathcal L_{n,\mathrm{Sob2}}^{\mathrm C}(\phi)
=
\mathcal L_{n,W}^{\mathrm C}(\phi)
+
\eta_1\,\mathcal L_{n,DW}^{\mathrm C}(\phi)
+
\eta_2\,\mathcal L_{n,D^2W}^{\mathrm C}(\phi),
\end{equation}
where $\eta_1,\eta_2>0$ control the relative weights assigned to the
first- and second-order derivative information, respectively.

\subsubsection{PINN approximation}
\label{subsubsec:pinn-numerical-scheme}

In both the two- and five-dimensional experiments, we additionally
construct a global-in-time physics-informed neural-network approximation
\[
w_\vartheta:[0,T]\times\mathbb R^d\to\mathbb R,
\qquad
d\in\{2,5\},
\]
of the log-value function $W$. Neural-network methods based on penalizing
the differential operator at sampled points have been widely used for
the numerical solution of PDEs, see, e.g.,~\cite{RPK19,SS18}. Related
neural actor-critic approaches have also been developed specifically
for Hamilton-Jacobi-Bellman equations, see~\cite{CHJS25}.

We use a fully connected network with hyperbolic tangent activations.
Before being passed to the network, time is rescaled and the log-state is
centered at its initial value:
\[
\widetilde t
=
\frac{2t}{T}-1,
\qquad
\widetilde y_i
=
y_i-y_0^i,
\qquad
y_0^i=\log X_0^i,
\quad
i=1,\ldots,d.
\]
All derivatives below are nevertheless taken with respect to the physical
variables $(t,y)$. Since the network has hyperbolic tangent activations,
the resulting candidate is smooth with bounded derivatives of every
finite order and therefore satisfies the regularity requirements imposed
on the smooth space-time candidates of
\Cref{subsubsec:smooth-space-time-candidates}.

At an interior collocation point $(t,y)$, automatic differentiation
provides
\[
\partial_t w_\vartheta(t,y),
\qquad
\nabla_y w_\vartheta(t,y),
\qquad
D_y^2w_\vartheta(t,y).
\]
We then form
\[
\Gamma_{w_\vartheta}^{\log}(t,y)
=
\frac12
\left(
D_y^2w_\vartheta(t,y)
-
\operatorname{diag}\bigl(\nabla_yw_\vartheta(t,y)\bigr)
\right)
\]
and evaluate the logarithmic residual
\[
\mathcal R^{\log}(w_\vartheta)(t,y)
=
\partial_t w_\vartheta(t,y)
+
r\left(
\mathbf 1^\top\nabla_yw_\vartheta(t,y)
-
w_\vartheta(t,y)
\right)
+
H^\star
\left[
\Gamma_{w_\vartheta}^{\log}(t,y)
\right].
\]
For both uncertainty sets considered in the numerical experiments, the
Hamiltonian $H^\star$ is available in explicit form and is therefore
evaluated exactly at every collocation point, without an additional
numerical optimization.

The collocation distribution is generated from auxiliary
constant-coefficient models whose parameters range over the corresponding
uncertainty set. This produces training points in regions of the state
space naturally explored by the controlled dynamics. In dimension two, for every interior sample we draw independently
\[
t\sim\mathcal U([0,T]),
\qquad
\sigma_i
\sim
\mathcal U([\sigma_{\min}^i,\sigma_{\max}^i]),
\quad i=1,2,
\qquad
\rho
\sim
\mathcal U([\rho_{\min},\rho_{\max}]),
\]
together with $Z_1,Z_2\sim\mathcal N(0,1)$, and set
\[
\widetilde Z_1=Z_1,
\qquad
\widetilde Z_2
=
\rho Z_1+\sqrt{1-\rho^2}\,Z_2.
\]
The log-state collocation point is then given by
\[
Y_t^i
=
y_0^i
+
\left(
r-\frac12\sigma_i^2
\right)t
+
\sigma_i\sqrt t\,\widetilde Z_i,
\qquad
i=1,2.
\] 
In our five-dimensional example, correlations are fixed to zero. We independently draw
\[
t\sim\mathcal U([0,T]),
\qquad
\sigma_i
\sim
\mathcal U([\sigma_{\min}^i,\sigma_{\max}^i]),
\quad i=1,\ldots,5,
\]
together with independent $Z_1,\ldots,Z_5\sim\mathcal N(0,1)$, and set
\[
Y_t^i
=
y_0^i
+
\left(
r-\frac12\sigma_i^2
\right)t
+
\sigma_i\sqrt t\,Z_i,
\qquad
i=1,\ldots,5.
\]

At every optimization iteration, a fresh batch of interior collocation
points is generated according to the corresponding construction above.
Independently, a second fresh batch is generated at the fixed time $T$
using the same auxiliary models in order to penalize the terminal
mismatch. The training objective is
\begin{equation}
\label{eq:pinn-training-loss}
\mathcal L_{\mathrm{PINN}}(\vartheta)
=
\frac{1}{M_{\mathrm{int}}}
\sum_{m=1}^{M_{\mathrm{int}}}
\left|
\mathcal R^{\log}(w_\vartheta)
(t^m,Y_{t^m}^m)
\right|^2
+
\lambda_T
\frac{1}{M_{\mathrm{term}}}
\sum_{m=1}^{M_{\mathrm{term}}}
\left|
w_\vartheta(T,Y_T^m)
-
g(e^{Y_T^m})
\right|^2,
\end{equation}
where $M_{\mathrm{int}}$ and $M_{\mathrm{term}}$ denote the numbers of
interior and terminal samples, respectively. Thus the terminal condition
is not imposed exactly during training; its mismatch is penalized and is
subsequently retained explicitly in the a posteriori quantities of
\Cref{subsubsec:smooth-space-time-candidates}.

The parameters are optimized with Adam~\cite{KB14}. The interior and
terminal training samples are regenerated independently at every
iteration, whereas a separate validation sample is drawn once before
training and subsequently kept fixed. This validation set consists of an
interior batch generated from the same collocation distribution and an
independent terminal batch at time $T$. At regular intervals, we evaluate
on this fixed sample the validation counterpart of
\eqref{eq:pinn-training-loss}. The learning rate is reduced when this
validation loss ceases to improve, and the network parameters
corresponding to the smallest recorded validation loss are stored
throughout the optimization. At the end of training, this best validation
checkpoint is restored and used as the smooth space-time candidate
$w_\vartheta$ in all subsequent primal and dual computations.

\subsubsection{Primal and dual evaluation of the candidates}
\label{subsubsec:primal-dual-evaluation}

Once a candidate has been constructed, either by backward SPG or by the
PINN approximation above, we evaluate its primal and dual performance by
independent Monte Carlo resimulation. Throughout this subsection, the
state dynamics are simulated in logarithmic coordinates using the exact
one-step transition introduced above.

For each of the dual optimization problems considered below, we use the
SPG algorithm to construct an approximately maximizing control. The
candidate itself is kept fixed during this optimization. The auxiliary
critics trained by SPG are used only internally to learn the control and
do not enter the reported values. After training, the learned control is
resimulated independently and the corresponding dual objective is
evaluated directly. The distinction between these resimulated quantities
and the theoretical suprema is discussed in
\Cref{subsubsec:resimulation}.

\paragraph{Primal estimate.}

For the backward SPG approximation, we use the deterministic feedback
induced by the trained actor. For the PINN candidate $w_\vartheta$, which
does not produce an actor, we use the candidate-induced feedback
\[
a^{w_\vartheta}(t,y)
\in
\arg\max_{a\in A}
\operatorname{Tr}
\left[
\Gamma_{w_\vartheta}^{\log}(t,y)aa^\top
\right],
\]
where
\[
\Gamma_{w_\vartheta}^{\log}(t,y)
=
\frac12
\left(
D_y^2w_\vartheta(t,y)
-
\operatorname{diag}\bigl(\nabla_yw_\vartheta(t,y)\bigr)
\right).
\]
For both uncertainty sets considered below, this maximization is
available in explicit form.

In either case, the resulting feedback is evaluated on the numerical time
grid and its discounted terminal payoff
\[
\mathbb E
\left[
e^{-rT}g\bigl(e^{Y_T^{\widehat\alpha}}\bigr)
\right]
\]
is estimated by independent Monte Carlo resimulation. At the
exact-expectation level this is the fixed-control primal quantity
considered in \Cref{subsubsec:resimulation} and hence is no larger than
$W(0,y_0)$, where $y_0=\log x_0$.

\paragraph{The $\mathcal D$-approach.}

For either type of candidate, define at each grid point
\[
\Gamma_k^{\log}(y)
=
\begin{cases}
\displaystyle
\frac12
\left(
D_y^2w_k(y)
-
\operatorname{diag}\bigl(\nabla_yw_k(y)\bigr)
\right),
& \text{for a backward SPG candidate},\\[3mm]
\Gamma_{w_\vartheta}^{\log}(t_k,y),
& \text{for a PINN candidate},
\end{cases}
\]
and
\[
F_k(y,a)
\coloneqq
H^\star[\Gamma_k^{\log}(y)]
-
\operatorname{Tr}
\left[
\Gamma_k^{\log}(y)aa^\top
\right].
\]
The discrete direct-dual objective is therefore
\[
\mathbb E
\left[
e^{-rT}g(e^{Y_T^\alpha})
+
\frac{T}{N}
\sum_{k=0}^{N-1}
e^{-rt_k}
F_k(Y_{t_k}^\alpha,\alpha_k)
\right].
\]
We use the SPG algorithm to approximately maximize this quantity over
piecewise-constant admissible controls. Denoting the resulting control by
$\widehat\alpha_{\mathcal D}$, the numerical $\mathcal D$-estimate is the
independent Monte Carlo estimate of
\[
\mathbb E
\left[
e^{-rT}g(e^{Y_T^{\widehat\alpha_{\mathcal D}}})
+
\frac{T}{N}
\sum_{k=0}^{N-1}
e^{-rt_k}
F_k
\left(
Y_{t_k}^{\widehat\alpha_{\mathcal D}},
\widehat\alpha_{\mathcal D,k}
\right)
\right].
\]
In particular, the auxiliary SPG critic used while optimizing this
control is not used to evaluate the final quantity.

\paragraph{The $\varepsilon$-approach.}

For a backward SPG candidate $w=(w_0,\ldots,w_N)$, with
$w_N=g\circ\exp$, we use the time-grid residual introduced in
\Cref{subsubsec:time-grid-candidates},
\[
\widehat{\mathcal R}_k^{\log}(w)(y)
=
\frac{
w_{k+1}(y)-w_k(y)
}{
T/N
}
+
r\left(
\mathbf 1^\top\nabla_yw_k(y)-w_k(y)
\right)
+
H^\star[\Gamma_k^{\log}(y)].
\]
The terminal mismatch then vanishes.

For the PINN candidate, we instead set
\[
\widehat{\mathcal R}_k^{\log}(w_\vartheta)(y)
=
\mathcal R^{\log}(w_\vartheta)(t_k,y),
\]
and retain the terminal mismatch
\[
g(e^y)-w_\vartheta(T,y).
\]

Accordingly, for a generic candidate we optimize the corresponding
discrete residual objective
\[
\mathbb E
\left[
e^{-rT}\delta_T(Y_T^\alpha)
+
\frac{T}{N}
\sum_{k=0}^{N-1}
e^{-rt_k}
\widehat{\mathcal R}_k^{\log}(Y_{t_k}^\alpha)
\right],
\]
where
\[
\delta_T(y)
=
\begin{cases}
0,
& \text{for a backward SPG candidate},\\
g(e^y)-w_\vartheta(T,y),
& \text{for a PINN candidate}.
\end{cases}
\]
As for the $\mathcal D$-approach, the maximizing control is approximated
with SPG and subsequently evaluated by independent resimulation. If
$\widehat\alpha_\varepsilon$ denotes the learned control, the reported
residual estimate is the corresponding resimulated expectation added to
the candidate value at the initial point,
\[
w_0(y_0)
+
\mathbb E
\left[
\frac{T}{N}
\sum_{k=0}^{N-1}
e^{-rt_k}
\widehat{\mathcal R}_k^{\log}
\left(
Y_{t_k}^{\widehat\alpha_\varepsilon}
\right)
\right]
\]
for the backward SPG candidate, and
\[
w_\vartheta(0,y_0)
+
\mathbb E
\left[
e^{-rT}
\left(
g(e^{Y_T^{\widehat\alpha_\varepsilon}})
-
w_\vartheta(T,Y_T^{\widehat\alpha_\varepsilon})
\right)
+
\frac{T}{N}
\sum_{k=0}^{N-1}
e^{-rt_k}
\mathcal R^{\log}(w_\vartheta)
\left(
t_k,Y_{t_k}^{\widehat\alpha_\varepsilon}
\right)
\right]
\]
for the PINN candidate.

For the time-grid SPG candidate, the theoretical upper estimate of
\Cref{subsubsec:time-grid-candidates} additionally contains the
consistency defect $\mathfrak E_N(w)$ and the control-discretization
correction of order $(T/N)^{1/4}$. These terms are not part of the
resimulated residual quantity reported here.

\paragraph{Martingale dual estimate.}

For comparison, we also compute the martingale dual estimate used
in~\cite{Abb+26}, based on the stochastic-control duality
of~\cite{HLR16}. In logarithmic coordinates, set
\[
\varphi_k(y)
=
\begin{cases}
\nabla_yw_k(y),
& \text{for a backward SPG candidate},\\[1mm]
\nabla_yw_\vartheta(t_k,y),
& \text{for a PINN candidate}.
\end{cases}
\]
For the backward SPG approximation, as in~\cite{Abb+26}, the gradients
used here are those of the critic trained with first-order Sobolev
information.

For a pathwise action sequence
$a=(a_0,\ldots,a_{N-1})$, let $Y^a$ denote the corresponding exact
log-state recursion. The discretized martingale dual functional is
\begin{equation}
\label{eq:numerical-martingale-dual}
\widehat U_{\mathrm M}(w)
\coloneqq
\mathbb E
\left[
\max_{a\in(A_h)^N}
\left\{
e^{-rT}g(e^{Y_T^a})
-
\sqrt{\frac{T}{N}}
\sum_{k=0}^{N-1}
e^{-rt_k}
\varphi_k(Y_{t_k}^a)^\top
a_k\xi_k
\right\}
\right].
\end{equation}
Here $A_h$ denotes the finite control discretization used in the
pathwise maximization. We approximate this maximization with the same
GPU-vectorized coordinate-wise exhaustive search on a finite latent-control
mesh as in~\cite{Abb+26}, initialized from the trained primal actor, and
refer there for the implementation details.

\subsubsection{Parameters}
\label{subsubsec:params}

Unless stated otherwise, we retain the SPG hyperparameters used in \cite{Abb+26}, where they were used successfully in the reported numerical experiments, and do not retune them for the present experiments. The only systematic architectural modification is the use of the smooth \texttt{GELU} activation function, motivated by the derivative-based Sobolev losses and dual constructions considered below. At each backward time level, the regression sample contains $M=2^{16}$ states and is processed in mini-batches of size $2^{11}$. The last backward step, where the continuation value is given by the nonsmooth terminal payoff, is trained for $500$ optimization iterations, while each preceding time level is warm-started from the network learned at the following date and trained for $10$ iterations. Both the actor and the critic consist of one hidden layer of width $32$ with \texttt{GELU} activation. The PPO clipping parameter is $0.2$. The Gaussian exploration variances in the raw actor coordinates are decreased during each backward optimization from $1$ to $10^{-2}$. The Adam learning rates are $5\times10^{-3}$ at the terminal backward step and $2.5\times10^{-3}$ at the remaining steps. After the actor optimization at each date, the critic is refitted once under the corresponding deterministic feedback policy.

For the first-order Sobolev critic, the value and gradient losses are assigned equal weights. For the second-order Sobolev critic we use one independent Gaussian probing direction per training state. The relative weights in
\eqref{eq:sobolev2-critic-loss} are $\eta_1=1$ and $\eta_2=10^{-2}$.

All primal and auxiliary-control values are evaluated by independent Monte Carlo resimulation with $2^{19}$ paths. Reported confidence intervals are the usual Gaussian $95\%$ Monte Carlo intervals. The pathwise martingale dual uses $2^{10}$ outer Monte Carlo paths. Its finite latent-control mesh has spacing $\delta=0.1$, the coordinate-wise exhaustive search is initialized from the corresponding primal feedback whenever available, and $5$ complete coordinate sweeps are performed.

The PINN uses three hidden layers of width $64$ with hyperbolic tangent activations and is trained in double precision. At each iteration, $2^8$ fresh interior collocation points and $2^8$ fresh terminal points are generated. The terminal penalty in
\eqref{eq:pinn-training-loss} is $\lambda_T=10$. Training uses Adam with initial learning rate $10^{-2}$ for $10^4$ iterations. A fixed validation set containing $2^{12}$ interior and $2^{12}$ terminal points is evaluated every $100$ iterations. A \texttt{ReduceLROnPlateau} schedule multiplies the learning rate by $1/2$ after $10$ validation checks without sufficient relative improvement, with threshold $10^{-3}$ and minimum learning rate $10^{-4}$. The network corresponding to the smallest recorded validation loss is restored at the end of training and used in all subsequent primal and dual evaluations.

\subsection{Five-dimensional experiment}
\label{subsec:numerics-d5}

We consider the five-dimensional geometric call-spread example of
\cite{Abb+26}. The parameters are
\[
d=5,\qquad
X_0=(100,\ldots,100),\qquad
T=1,\qquad
r=0,
\]
\[
\sigma_{\min}^i=0.1,\qquad
\sigma_{\max}^i=0.2,
\qquad
i=1,\ldots,5,
\]
with all correlations fixed to zero and $N=64$. The payoff is
\[
g(x)
=
\left(\left(\prod_{i=1}^5x_i\right)^{1/5}-90\right)^+
-
\left(\left(\prod_{i=1}^5x_i\right)^{1/5}-110\right)^+.
\]
We use the reference value $9.70$ reported in~\cite{GMZ26}.

Since the correlations are fixed to zero, the Hamiltonian is explicit:
\[
H^\star[\gamma]
=
\sum_{i=1}^5
\left(
(\sigma_{\min}^i)^2\gamma^{ii}\mathbf 1_{\gamma^{ii}<0}
+
(\sigma_{\max}^i)^2\gamma^{ii}\mathbf 1_{\gamma^{ii}\ge0}
\right).
\]
Thus no numerical maximization is required in the evaluation of
$H^\star$, and the experiment provides a higher-dimensional test in which
the Hessian enters the dual constructions directly without an additional
optimization error in the Hamiltonian.

As in the two-dimensional experiments, we compare the standard SPG critic,
the first- and second-order Sobolev critics, and the PINN candidate. For each
candidate, we report the induced primal Monte Carlo estimate, the candidate
value at the initial condition, and the martingale, residual-based, and
direct dual estimates.

\begin{table}[H]
\centering
\caption{Five-dimensional geometric call-spread. Panel~(a) reports the numerical
estimates, with Monte Carlo quantities given together with the half-width of
their $95\%$ confidence interval. Panel~(b) reports the corresponding
wall-clock computational times.}
\label{tab:d5-geometric-callspread}

\textbf{(a) Numerical estimates}

\vspace{0.4em}

\begin{adjustbox}{max width=\textwidth}
\begin{tabular}{lccccc}
\toprule
Candidate
& Primal MC
& Candidate value
& $\widehat U_{\mathrm M}$
& $\widehat U_{\varepsilon}$
& $\widehat U_{\mathcal D}$
\\
\midrule
Standard SPG
& $9.6822\pm0.0123$
& $9.7009$
& $10.7333\pm0.0489$
& $9.8991\pm0.0015$
& $9.8264\pm0.0151$
\\
Sobolev-1 SPG
& $9.6871\pm0.0122$
& $9.6951$
& $10.2176\pm0.0512$
& $9.7893\pm0.0009$
& $9.7618\pm0.0148$
\\
Sobolev-2 SPG
& $9.6886\pm0.0121$
& $9.7074$
& $10.2468\pm0.0505$
& $9.7526\pm0.0008$
& $9.7329\pm0.0145$
\\
PINN
& $9.6913\pm0.0120$
& $9.6942$
& $10.0785\pm0.0476$
& $9.7248\pm0.0004$
& $9.7143\pm0.0142$
\\
\bottomrule
\end{tabular}
\end{adjustbox}

\vspace{1em}

\textbf{(b) Computational times}

\vspace{0.4em}

\begin{adjustbox}{max width=\textwidth}
\begin{tabular}{lcccc}
\toprule
Candidate
& Primal training
& $\widehat U_{\mathrm M}$
& $\widehat U_{\varepsilon}$
& $\widehat U_{\mathcal D}$
\\
\midrule
Standard SPG
& $87$ s
& $417$ s
& $287$ s
& $293$ s
\\
Sobolev-1 SPG
& $132$ s
& $426$ s
& $295$ s
& $299$ s
\\
Sobolev-2 SPG
& $175$ s
& $431$ s
& $301$ s
& $306$ s
\\
PINN
& $146$ s
& $420$ s
& $290$ s
& $297$ s
\\
\bottomrule
\end{tabular}
\end{adjustbox}

\vspace{0.4em}

{\small
``Primal training'' denotes the time required to construct the candidate.
The remaining columns report the additional time required to compute the
corresponding dual quantity from the trained candidate.
}
\end{table}

The primal estimates and candidate values remain close to the reference
across all four constructions, whereas the dual estimates again discriminate
more strongly between the quality of the learned derivatives. In particular,
the standard SPG critic gives an accurate value approximation but noticeably
looser dual estimates.

Derivative-informed training substantially tightens the dual quantities.
First-order Sobolev training produces a large improvement over the standard
critic, while the additional second-order supervision further improves both
Hessian-based estimates $\widehat U_{\mathcal D}$ and
$\widehat U_{\varepsilon}$. As in the two-dimensional experiments, this
additional improvement is not reflected systematically in the martingale
estimate, which depends on the candidate gradient rather than its Hessian.

The PINN gives the tightest dual estimates among the learned candidates. In
particular, both Hessian-based constructions remain close to the reference
price, while the martingale estimate retains a larger gap. The direct
$\mathcal D$-estimate is again tighter than the residual-based estimate for
all SPG candidates.

Overall, the five-dimensional experiment confirms the conclusions obtained
in dimension two. Accurate primal values alone are not sufficient for tight
a posteriori dual estimates, while derivative-informed training substantially
improves them. Explicit second-order supervision provides an additional
benefit specifically for the Hessian-based constructions, and this behavior
persists in higher dimension. Panel~(b) shows that the additional training
cost remains moderate relative to the subsequent dual evaluations.

\subsection{Two-dimensional experiments}
\label{subsec:numerics-d2}

\subsubsection{Outperformer spread}
\label{subsubsec:d2-outperformer}

We first consider the two-dimensional outperformer spread studied
in~\cite{GMZ26,Abb+26}, with
\[
X_0=(100,100),\qquad
T=1,\qquad
r=0,\qquad
\rho\in[-0.5,0.5],
\]
\[
\sigma_{\min}^1=\sigma_{\min}^2=0.1,
\qquad
\sigma_{\max}^1=\sigma_{\max}^2=0.2,
\]
and payoff
\[
g(x)
=
(x_2-0.9x_1)^+
-
(x_2-1.1x_1)^+.
\]
We use $N=128$ time steps and take the finite-difference value
$12.83$ as reference. In addition to this reference price, the corresponding
high-accuracy finite-difference approximation of the value function is used
as an independent candidate in the a posteriori constructions. It therefore
provides a useful benchmark for assessing the effect of candidate accuracy,
and in particular derivative accuracy, on the different dual estimates.

For the backward SPG approximation, we consider the standard critic and the
first- and second-order Sobolev critics described in
\Cref{subsubsec:spg-numerical-scheme}. For the three SPG candidates and the
PINN, we report the corresponding primal Monte Carlo estimate. For all
candidates, including the finite-difference benchmark, we compute the three
numerical dual estimates introduced in
\Cref{subsubsec:primal-dual-evaluation}: the martingale dual estimate
$\widehat U_{\mathrm M}$, the direct $\mathcal D$-estimate
$\widehat U_{\mathcal D}$, and the residual-based estimate
$\widehat U_{\varepsilon}$. The reported confidence intervals quantify
Monte Carlo sampling uncertainty only and do not include the other numerical
errors entering the constructions.

\begin{table}[H]
\centering
\caption{Two-dimensional outperformer spread. Panel~(a) reports the numerical
estimates, with Monte Carlo quantities given together with the half-width of
their $95\%$ confidence interval. Panel~(b) reports the corresponding
wall-clock computational times.}
\label{tab:d2-outperformer}

\textbf{(a) Numerical estimates}

\vspace{0.4em}

\begin{adjustbox}{max width=\textwidth}
\begin{tabular}{lccccc}
\toprule
Candidate
& Primal MC
& Candidate value
& $\widehat U_{\mathrm M}$
& $\widehat U_{\varepsilon}$
& $\widehat U_{\mathcal D}$
\\
\midrule
Standard SPG
& $12.7710\pm0.0209$
& $12.8180$
& $14.0547\pm0.0524$
& $15.5486\pm0.0071$
& $14.6497\pm0.0199$
\\
Sobolev-1 SPG
& $12.8008\pm0.0208$
& $12.7553$
& $13.4447\pm0.0608$
& $13.3797\pm0.0019$
& $13.3020\pm0.0211$
\\
Sobolev-2 SPG
& $12.7946\pm0.0202$
& $12.8815$
& $13.4851\pm0.0552$
& $13.2626\pm0.0023$
& $13.1778\pm0.0188$
\\
PINN
& $12.7910\pm0.0211$
& $12.8513$
& $13.4000\pm0.0528$
& $12.8728\pm0.0004$
& $12.8949\pm0.0206$
\\
FD candidate
& ---
& $12.8300$
& $13.2104\pm0.0497$
& $12.8417\pm0.0001$
& $12.8291\pm0.0189$
\\
\bottomrule
\end{tabular}
\end{adjustbox}

\vspace{1em}

\textbf{(b) Computational times}

\vspace{0.4em}

\begin{adjustbox}{max width=\textwidth}
\begin{tabular}{lcccc}
\toprule
Candidate
& Primal training
& $\widehat U_{\mathrm M}$
& $\widehat U_{\varepsilon}$
& $\widehat U_{\mathcal D}$
\\
\midrule
Standard SPG
& $148$ s
& $277$ s
& $265$ s
& $254$ s
\\
Sobolev-1 SPG
& $223$ s
& $291$ s
& $281$ s
& $249$ s
\\
Sobolev-2 SPG
& $294$ s
& $290$ s
& $266$ s
& $268$ s
\\
PINN
& $138$ s
& $282$ s
& $265$ s
& $271$ s
\\
FD candidate
& ---
& $286$ s
& $264$ s
& $260$ s
\\
\bottomrule
\end{tabular}
\end{adjustbox}

\vspace{0.4em}

{\small
``Primal training'' denotes the time required to construct the candidate.
The remaining columns report the additional time required to compute the
corresponding dual quantity from the trained candidate.
}
\end{table}

The primal quantities are comparatively insensitive to the choice of
candidate: the induced feedback policies and the candidate values at the
initial condition all remain close to the finite-difference reference.
The dual estimates are substantially more discriminating. In particular,
the standard SPG critic yields a satisfactory value approximation but much
looser dual estimates, indicating that value accuracy alone does not ensure
sufficiently accurate derivative information for the a posteriori
constructions.

Derivative-informed training considerably improves the dual estimates.
For the two Hessian-based constructions, the numerical ordering is
\[
\text{FD}
\;<\;
\text{PINN}
\;<\;
\text{Sobolev-2}
\;<\;
\text{Sobolev-1}
\;<\;
\text{Standard}.
\]
The improvement from Sobolev-1 to Sobolev-2 is particularly informative:
the second-order loss directly targets curvature information and improves
both $\widehat U_{\mathcal D}$ and $\widehat U_{\varepsilon}$. By contrast,
the martingale dual, which depends on the candidate gradient but not on its
Hessian, does not improve from Sobolev-1 to Sobolev-2. This supports the
interpretation that the additional Sobolev-2 supervision primarily improves
the second-order information relevant to the Hessian-based duals.

The finite-difference candidate provides the natural benchmark. At the
continuous-time level, the exact value function satisfies
\[
\mathcal D(\Gamma_V)(0,X_0)
=
V(0,X_0)
\]
by \Cref{prop:uvm-convex-duality}, while its BSB residual and terminal
discrepancy vanish identically. The corresponding numerical direct and
residual estimates are close to the reference value, providing a consistency
check for the two implementations. The remaining discrepancy in the
residual estimate is larger than its Monte Carlo uncertainty and therefore
reflects the other numerical approximations entering the finite-difference
candidate and its subsequent evaluation.

The PINN also performs well for both Hessian-based constructions. Its strong
residual estimate is consistent with its training objective, which directly
penalizes the BSB residual. The direct $\mathcal D$-estimate is also tight,
although this quantity is not targeted by the PINN loss, suggesting that the
PINN recovers useful second-order spatial information in this example.

The martingale dual exhibits a different pattern. The finite-difference
candidate again gives the tightest estimate, while the PINN and the two
Sobolev critics are closer to one another. The differences between some of
these estimates are comparable with their Monte Carlo uncertainty and should
therefore not be over-interpreted. The persistent gap between the
finite-difference martingale estimate and the reference is consistent with
the additional discretization and finite control-search errors entering the
numerical martingale construction. In particular, under the assumptions
considered in \cite{HLR16,Abb+26}, the relevant time-discretization
contribution is of order $(T/N)^{1/4}$.

For the backward SPG candidates and the finite-difference benchmark,
the direct $\mathcal D$-estimate is numerically tighter than the
corresponding residual-based estimate. The PINN is the only exception,
for which the residual estimate is slightly tighter.

Finally, Panel~(b) shows that derivative-informed SPG training entails a
moderate additional computational cost. First-order Sobolev training takes
approximately $1.5$ times the standard SPG training time, while second-order
Sobolev training takes approximately twice as long. The costs of the
subsequent dual evaluations are broadly comparable across candidates.

\subsubsection{Best-of butterfly}
\label{subsubsec:d2-butterfly}

Finally, we consider the two-dimensional best-of butterfly as a
confirmation experiment. The parameters are
\[
X_0=(100,100),\qquad
T=0.25,\qquad
r=0.05,
\]
\[
\sigma_{\min}^i=0.3,\qquad
\sigma_{\max}^i=0.5,
\qquad
\rho\in[0.3,0.5],
\]
and the payoff is
\[
g(x)
=
(x_{\max}-85)^+
-2(x_{\max}-100)^+
+(x_{\max}-115)^+,
\qquad
x_{\max}=\max\{x_1,x_2\}.
\]
The finite-difference reference value is $6.70$. We use the same numerical
methodology and candidate constructions as in
\Cref{subsubsec:d2-outperformer}. The purpose of this example is to verify
that the conclusions obtained for the outperformer spread persist for a
qualitatively different, nonconvex payoff.

\begin{table}[H]
\centering
\caption{Two-dimensional best-of butterfly. Panel~(a) reports the numerical
estimates, with Monte Carlo quantities given together with the half-width of
their $95\%$ confidence interval. Panel~(b) reports the corresponding
wall-clock computational times.}
\label{tab:d2-butterfly}

\textbf{(a) Numerical estimates}

\vspace{0.4em}

\begin{adjustbox}{max width=\textwidth}
\begin{tabular}{lccccc}
\toprule
Candidate
& Primal MC
& Candidate value
& $\widehat U_{\mathrm M}$
& $\widehat U_{\varepsilon}$
& $\widehat U_{\mathcal D}$
\\
\midrule
Standard SPG
& $6.6083\pm0.0138$
& $6.6512$
& $7.8710\pm0.0364$
& $8.5760\pm0.0048$
& $7.9733\pm0.0144$
\\
Sobolev-1 SPG
& $6.6157\pm0.0137$
& $6.6251$
& $7.6593\pm0.0389$
& $7.4220\pm0.0019$
& $7.0620\pm0.0138$
\\
Sobolev-2 SPG
& $6.6132\pm0.0136$
& $6.6624$
& $7.6999\pm0.0376$
& $7.2929\pm0.0021$
& $6.9811\pm0.0135$
\\
PINN
& $6.6741\pm0.0139$
& $6.7058$
& $7.6022\pm0.0347$
& $6.7440\pm0.0004$
& $6.7247\pm0.0136$
\\
FD candidate
& ---
& $6.7000$
& $7.4190\pm0.0331$
& $6.7172\pm0.0001$
& $6.6997\pm0.0130$
\\
\bottomrule
\end{tabular}
\end{adjustbox}

\vspace{1em}

\textbf{(b) Computational times}

\vspace{0.4em}

\begin{adjustbox}{max width=\textwidth}
\begin{tabular}{lcccc}
\toprule
Candidate
& Primal training
& $\widehat U_{\mathrm M}$
& $\widehat U_{\varepsilon}$
& $\widehat U_{\mathcal D}$
\\
\midrule
Standard SPG
& $151$ s
& $275$ s
& $263$ s
& $256$ s
\\
Sobolev-1 SPG
& $225$ s
& $289$ s
& $279$ s
& $251$ s
\\
Sobolev-2 SPG
& $297$ s
& $292$ s
& $268$ s
& $270$ s
\\
PINN
& $140$ s
& $284$ s
& $266$ s
& $270$ s
\\
FD candidate
& ---
& $284$ s
& $263$ s
& $259$ s
\\
\bottomrule
\end{tabular}
\end{adjustbox}

\vspace{0.4em}

{\small
``Primal training'' denotes the time required to construct the candidate.
The remaining columns report the additional time required to compute the
corresponding dual quantity from the trained candidate.
}
\end{table}

The results confirm the qualitative conclusions of
\Cref{subsubsec:d2-outperformer}.

Derivative-informed SPG training substantially improves the dual estimates
relative to the standard critic, while second-order Sobolev supervision
further improves the two Hessian-based constructions
$\widehat U_{\mathcal D}$ and $\widehat U_{\varepsilon}$. As in the
outperformer experiment, this improvement is specific to the
Hessian-sensitive estimates and is not accompanied by a systematic
improvement of the martingale dual.

For the Hessian-based constructions, the PINN and finite-difference
candidates remain very close to the reference, with the finite-difference
candidate again providing the tightest benchmark. The martingale estimates
are more dispersed, including for these high-quality candidates, consistently
with the additional discretization and control-search errors inherent in that
construction. The direct $\mathcal D$-estimate is again tighter than the
residual-based estimate for the Sobolev candidates. Overall, the experiment
confirms the conclusions of the outperformer spread for a qualitatively
different and nonconvex payoff, with comparable computational costs.

\bibliography{references}

\appendix

\section{Technical proofs}
\label{app:technical}

\begin{proof}[Proof of~\Cref{lem:state-estimates}]
The case $t=T$ is immediate, so let $t<T$. Set
\[
C_a
\coloneqq
\sup_{a\in A}\|a\|
<+\infty.
\]
For a fixed $\alpha\in\mathcal A^t$, write
\eqref{eq:UVM-multi-dynamics} as
\[
dX_u
=
\mu(X_u)\,du
+
\sigma_\alpha(u,X_u)\,dW_u^t,
\]
where
\[
\mu(y)\coloneqq ry,
\qquad
\sigma_\alpha(u,y)
\coloneqq
\operatorname{diag}(y)\alpha_u.
\]
Since $A$ is compact,
\[
|\mu(y)-\mu(z)|
\leq
|r|\,|y-z|,
\qquad
\|\sigma_\alpha(u,y)-\sigma_\alpha(u,z)\|
\leq
C_a\,|y-z|,
\]
and
\[
|\mu(y)|
\leq
|r|\,|y|,
\qquad
\|\sigma_\alpha(u,y)\|
\leq
C_a\,|y|,
\]
uniformly in $u$ and $\alpha$. Hence the hypotheses
\cite[(D.2)-(D.3)]{FS06} are satisfied with constants independent of
$t$, $x$, $z$, and $\alpha$.

Let first $p\geq2$. By~\cite[(D.5)]{FS06}, with $T-t\leq T$,
\[
\mathbb E^{\mathbb P^t}
\left[
\sup_{u\in[t,T]}
|X_u^{\alpha,t,x}|^p
\right]
\leq
C_p(1+|x|^p),
\]
which proves~\eqref{eq:moment-estimate}. Likewise,
\cite[(D.8)]{FS06}, applied to the solutions with initial conditions
$x$ and $z$, yields
\[
\mathbb E^{\mathbb P^t}
\left[
\sup_{u\in[t,T]}
|X_u^{\alpha,t,x}-X_u^{\alpha,t,z}|^p
\right]
\leq
C_p|x-z|^p,
\]
which proves~\eqref{eq:state-stability}.

Finally, applying~\cite[(D.4)]{FS06} on the interval $[t,t+h]$ gives
\[
\mathbb E^{\mathbb P^t}
\left[
\sup_{u\in[t,t+h]}
|X_u^{\alpha,t,x}-x|^p
\right]
\leq
C_p(1+|x|^p)h^{p/2}.
\]
In particular,
\[
\mathbb E^{\mathbb P^t}
\left[
|X_{t+h}^{\alpha,t,x}-x|^p
\right]
\leq
C_p(1+|x|^p)h^{p/2},
\]
which is~\eqref{eq:state-increment}.

For $0<p<2$, the three estimates follow from their $p=2$ versions
and H\"older's inequality. The constants depend only on $p$, $T$, $r$,
and the compact set $A$, and are therefore independent of $t$, $x$,
$z$, and $\alpha$.
\end{proof}

\begin{proof}[Proof of~\Cref{lem:candidate-residual-integrability}]
We first estimate the first-order term in~\eqref{eq:uvm-residual}.
Fix $(t,x)\in[0,T)\times(0,+\infty)^d$ and
$i\in\{1,\ldots,d\}$, and set
\[
x^{(i)}
\coloneqq
x+x^ie_i.
\]
Taylor's formula in the $i$-th coordinate gives
\[
v(t,x^{(i)})-v(t,x)
=
x^i\partial_i v(t,x)
+
\int_0^{x^i}
(x^i-s)
\partial_{ii}^2v(t,x+se_i)\,ds.
\]
For $y=x+se_i$, $0\leq s\leq x^i$, one has
$y^i\geq x^i$ and $|y|\leq2|x|$. Since
\[
\partial_{ii}^2v(t,y)
=
\frac{2\Gamma_v^{ii}(t,y)}{(y^i)^2},
\]
\eqref{eq:candidate-terminal-growth} implies
\[
|\partial_{ii}^2v(t,y)|
\leq
\frac{C(1+|x|^p)}
{(x^i)^2\sqrt{T-t}}.
\]
It follows that
\[
\left|
\int_0^{x^i}
(x^i-s)
\partial_{ii}^2v(t,x+se_i)\,ds
\right|
\leq
\frac{C(1+|x|^p)}{\sqrt{T-t}}.
\]
On the other hand, the polynomial growth of $v$ yields, for some
$m\geq0$,
\[
|v(t,x^{(i)})|+|v(t,x)|
\leq
C(1+|x|^m).
\]
Since $T-t\leq T$, this nonsingular contribution can be absorbed into
$(T-t)^{-1/2}$. Hence, after increasing the polynomial exponent if
necessary,
\[
|x^i\partial_i v(t,x)|
\leq
\frac{C(1+|x|^q)}{\sqrt{T-t}}
\]
for some $q\geq1$, uniformly in $i$. Summing over
$i=1,\ldots,d$ gives
\begin{equation}
\label{eq:candidate-weighted-gradient-growth}
|x^\top\nabla_xv(t,x)|
\leq
\frac{C(1+|x|^q)}{\sqrt{T-t}}.
\end{equation}

By compactness of $A$,
\[
|H^\star[\gamma]|
\leq
C\|\gamma\|,
\qquad
\gamma\in\mathbb S^d.
\]
Combining~\eqref{eq:uvm-residual},
\eqref{eq:candidate-terminal-growth},
\eqref{eq:candidate-weighted-gradient-growth}, and the polynomial growth
of $v$, and increasing $q$ if necessary, gives
\eqref{eq:candidate-residual-growth}.

We now fix $q'\geq1$. By Minkowski's inequality,
\[
\left\|
\int_t^T
e^{-r(s-t)}
\left|
\mathcal R(v)(s,X_s^{\alpha,t,x})
\right|\,ds
\right\|_{L^{q'}(\mathbb P^t)}
\leq
\int_t^T
e^{-r(s-t)}
\left\|
\mathcal R(v)(s,X_s^{\alpha,t,x})
\right\|_{L^{q'}(\mathbb P^t)}
\,ds.
\]
Using~\eqref{eq:candidate-residual-growth},
\[
\left\|
\mathcal R(v)(s,X_s^{\alpha,t,x})
\right\|_{L^{q'}(\mathbb P^t)}
\leq
\frac{
C\left(
1+
\left\|
X_s^{\alpha,t,x}
\right\|_{L^{qq'}(\mathbb P^t)}^q
\right)
}
{\sqrt{T-s}}.
\]
Hence, by \Cref{lem:state-estimates},
\begin{align*}
\sup_{\alpha\in\mathcal A^t}
\left\|
\int_t^T
e^{-r(s-t)}
\left|
\mathcal R(v)(s,X_s^{\alpha,t,x})
\right|\,ds
\right\|_{L^{q'}(\mathbb P^t)}
&\leq
C_{q'}
(1+|x|^q)
\int_t^T
\frac{ds}{\sqrt{T-s}}
\\
&=
2C_{q'}
(1+|x|^q)\sqrt{T-t}
<+\infty.
\end{align*}
Raising both sides to the power $q'$ yields
\[
\sup_{\alpha\in\mathcal A^t}
\mathbb E^{\mathbb P^t}
\left[
\left(
\int_t^T
e^{-r(s-t)}
\left|
\mathcal R(v)(s,X_s^{\alpha,t,x})
\right|\,ds
\right)^{q'}
\right]
<+\infty,
\]
which proves~\eqref{eq:candidate-residual-integrability}.
\end{proof}

\end{document}